\documentclass[11pt]{article}
\usepackage{amsmath,amssymb,amsthm,mathtools,hyperref}
\usepackage[margin=1in]{geometry}
\usepackage{enumerate}
\usepackage{xurl}
\usepackage{needspace}
\usepackage{titlesec}
\usepackage{color}
\usepackage{listings}
\lstdefinestyle{verificationpython}{
  language=Python,
  basicstyle=\ttfamily\footnotesize,
  columns=fullflexible,
  keepspaces=true,
  showstringspaces=false,
  breaklines=true,
  breakatwhitespace=true,
  breakautoindent=true,
  breakindent=1.5em,
  numbers=left,
  numberstyle=\tiny,
  numbersep=8pt,
  xleftmargin=2em,
  tabsize=4
}

\newtheorem{theorem}{Theorem}[section]
\newtheorem{thm}[theorem]{Theorem}

\newtheorem{prop}[theorem]{Proposition}

\newtheorem{lemma}[theorem]{Lemma}
\newtheorem{lem}[theorem]{Lemma}

\newtheorem{cor}[theorem]{Corollary}

\newtheorem{rem}[theorem]{Remark}

\theoremstyle{definition}
\newtheorem{defn}{Definition}[section]

\titleformat{\subsection}[runin]{\bfseries}{}{0pt}{}[.]
\titlespacing*{\subsection}{0pt}{1.0ex plus .3ex minus .2ex}{0.6em}
\titleformat{\paragraph}[runin]{\itshape}{}{0pt}{}[.]
\titlespacing*{\paragraph}{0pt}{0.8ex plus .2ex minus .2ex}{0.6em}

\newcommand{\F}{\mathbb{F}}
\newcommand{\GR}{\mathrm{GR}}
\newcommand{\A}{\mathbb{A}}
\newcommand{\bP}{\mathbb{P}}

\DeclareMathOperator{\Tr}{Tr}

\title{\bf Reduced APN permutations and a disproof of the Galois-ring lifting conjecture}
\author{{\Large\bf Daniele Bartoli$^1$, Pantelimon St\u{a}nic\u{a}$^2$}
\vspace{.4cm}\\
$^1$ Department of Mathematics and Computer Science,\\
University of Perugia, 06123 Perugia, Italy,
\texttt{daniele.bartoli@unipg.it}\\
$^2$ Department of Applied Mathematics, Naval Postgraduate School,\\
Monterey, CA 93943, USA, \texttt{pstanica@nps.edu}}
\date{\today}
\allowdisplaybreaks

\begin{document}
\maketitle

\begin{abstract}
We construct infinite families of reduced almost perfect nonlinear (APN) permutations of $\F_{2^m}$ that admit coefficientwise permutation lifts to $\GR(2^k,m)$ for every $k>1$. These families disprove Conjecture~2 of R\o njom and Sandrib (Cryptogr. Commun., 2026), even for reduced polynomial representatives. For every odd $m\ge5$ with $3\nmid m$, we obtain explicit seven-term polynomials over $\F_2$ with formal derivative identically one. We derive them from a complete criterion for the reduced representative of $A(x)^{2^s+1}$ to have no critical point in $\F_{2^m}$, where $A$ is an invertible linearized polynomial and $\gcd(s,m)=1$. We then apply the cube case to construct a second family of degree $2^{m-1}+2$, whose nonconstant derivatives avoid zero by an explicit trace identity. A fixed parameter in $\F_{32}$ gives this family in every dimension $m=5r$ with $r$ odd, including a degree-$18$ example over $\F_{32}$. Both constructions remain in Gold linear equivalence classes, so liftability of reduced representatives is not invariant under invertible $\F_2$-linear input changes. We obtain the lifting conclusions from the standard Galois-ring permutation criterion. We also verify a degree-$24$ example supplied by Zhang and Haobo, prove a trace-adjoint criterion for reduced quadratic functions, and interpret the derivative condition through ramification at rational points.
\end{abstract}

\section{Introduction}

Almost perfect nonlinear (APN) functions over finite fields of characteristic two arise in symmetric cryptography, where they minimize the largest differential probability of a single substitution box. Standard references for APN functions and related cryptographic criteria are Budaghyan~\cite{budaghyan} and Carlet~\cite{carlet}.
Throughout, we assume $m\ge2$ when discussing APN functions over $\F_{2^m}$ and put $q=2^m$ unless specified otherwise. For a field $K$, we write $K^\times=K\setminus\{0\}$. For a finite set $S$, we denote its cardinality by $|S|$.
\Needspace{12\baselineskip}
\begin{defn}Let $f : \F_{2^m} \to \F_{2^m}$ be a function. For $a \in \F_{2^m}$, define the \emph{directional derivative}
\[
D_a f(x) := f(x+a) + f(x).
\]
The \emph{differential uniformity} of $f$ is
\[
\delta_f := \max_{\substack{a \in \F_{2^m}^{\times} \\ b \in \F_{2^m}}} \left|\{x \in \F_{2^m} : D_a f(x) = b\}\right|.
\]
We say $f$ is \emph{almost perfect nonlinear} (APN) if $\delta_f = 2$.
\end{defn}

The value $\delta_f = 2$ is optimal over fields of characteristic two because $D_a f(x)=D_a f(x+a)$ whenever $a\ne0$. We call a solution set $\{x:D_af(x)=b\}$ for fixed $a\ne0$ and $b$ a differential fiber. Every nonempty such fiber therefore has even cardinality. A map is $\F_2$-affine if it has the form $x\mapsto L(x)+c$ with $L$ an $\F_2$-linear map. Its directional derivatives are constant and its differential uniformity is $2^m$. Thus APN functions attain the minimum differential uniformity in characteristic two.

When an APN function is also a permutation of $\F_{2^m}$, it is useful in the design of substitution boxes (S-boxes), where both nonlinearity and invertibility are required. The construction of APN permutations is particularly restrictive in even dimension. Dillon's permutation over $\F_{2^6}$ is an even-dimensional example~\cite{BrowningDillon10}. The inverse function is APN only in odd dimension. For odd $m$ several infinite monomial families exist, including the Gold, Kasami, Welch, Niho, inverse, and Dobbertin families. A \emph{Gold power} is a map $x\mapsto x^{2^s+1}$ with an integer $s\ge1$ satisfying $\gcd(s,m)=1$. These maps are APN and, for odd $m$, are permutations. See the surveys in~\cite{budaghyan,carlet}. We study the additional constraints that arise when we lift their polynomial representatives to Galois rings.

We ask whether permutation properties persist when one lifts polynomials from finite fields to Galois rings. For an element $a$ of a commutative ring $R$, $(a)=aR$ denotes its principal ideal. A ring is \emph{local} if it has a unique maximal ideal, and its \emph{residue field} is the quotient by that ideal. For integers $k,m\ge1$, the Galois ring $\GR(2^k,m)$ is, up to isomorphism, the unique finite commutative local ring of characteristic $2^k$ and cardinality $2^{km}$ whose maximal ideal is $(2)$ and whose residue field is $\F_{2^m}$. We refer to McDonald and Wan for background and constructions~\cite{mcdonald,Wan03}. If $f\in\F_{2^m}[x]$, a \emph{coefficientwise lift} of $f$ to $\GR(2^k,m)$ is a polynomial $F\in\GR(2^k,m)[x]$ whose coefficients reduce modulo $(2)$ to those of $f$. We also call $F$ a lift of $f$.

Throughout, if $f(x)=\sum_{i=0}^n a_i x^i$, then
$f'(x)=\sum_{i=1}^n i a_i x^{i-1}$ denotes its \emph{formal derivative}. A polynomial over a finite field is called a \emph{permutation polynomial} if the function it induces on that field is bijective. For a polynomial $f\in k[x]$, a point $a$ in an extension field of $k$ is called a \emph{critical point} if $f'(a)=0$. We say that $f$ is \emph{nowhere critical over $k$} if it has no critical point in $k$. A critical point in $k$ itself is called $k$-rational.

Every function $\F_q\to\F_q$ has a unique polynomial representative of degree less than $q$, by finite-field interpolation. See Lidl--Niederreiter~\cite{LidlNiederreiter97}. We call it the \emph{reduced representative}. For $P\in\F_q[x]$, we write $\mathcal{R}_q(P)$ for this representative, equivalently the remainder of $P$ modulo $x^q-x$. This convention is essential in a lifting problem because two polynomials that induce the same function on $\F_q$ can have different formal derivatives and different coefficientwise lifts.

 The Frobenius automorphism of $\F_q$ is $x\mapsto x^2$, and its iterates are $x\mapsto x^{2^i}$. A polynomial $A(x)=\sum_{i=0}^{m-1}a_ix^{2^i}$ is called \emph{linearized}, since it induces an $\F_2$-linear map. Every $\F_2$-linear map on $\F_q$ has a unique representation of this form~\cite{LidlNiederreiter97}. For a polynomial $A$, the notation $A^d$ means its ordinary $d$th power. Composition of maps is denoted by $\circ$.

R\o njom and Sandrib~\cite[Conjecture~2]{RS25} conjectured that an APN permutation polynomial over $\F_{2^m}$ cannot be lifted to a permutation polynomial over $\GR(2^k,m)$ for $k>1$. The APN requirement in this statement concerns the residue-field function. The standard permutation criterion recalled in their paper~\cite[Theorem~2]{RS25} says that a residue permutation admits a permutation lift precisely when its formal derivative has no zero in the residue field. In that case every coefficientwise lift is a permutation for every $k>1$.

The conjecture is false, even if one restricts to reduced representatives. Zhang and Haobo supplied a degree-$24$ reduced APN permutation over $\F_{32}$ with nowhere-zero formal derivative, together with a certificate through the TrureTuring project~\cite{ZhangHaobo26}. This example already yields a counterexample to Conjecture~2 by the standard lifting criterion. We give a self-contained verification of its finite-field properties and derive its lifting consequences in Section~\ref{sec:counterexample}. We use this example as a starting point for constructing reduced counterexamples in infinitely many dimensions. Our main contribution is a general coefficient criterion that yields two explicit infinite families with algebraic proofs.

We answer this question with two explicit constructions. The first has coefficients in $\F_2$ and constant formal derivative one. Let $q=2^m$, where $m\ge5$ is odd and $3\nmid m$. We prove that
\[
P_m(x)=x+x^4+x^{10}+x^{q/8+2}+x^{q/8+8}+x^{5q/8}+x^{q/2+8}
\]
is a reduced APN permutation of $\F_q$. Every exponent other than one is even, so $P_m'=1$. Consequently every coefficientwise lift permutes $\GR(2^k,m)$ for every $k>1$. To prove the finite-field properties, we exhibit an invertible $\F_2$-linear map $A_m$ such that $P_m(x)=A_m(x)^5$ on $\F_q$. The construction applies in every dimension $m\equiv1,5\pmod6$ with $m\ge5$.

The second construction uses the absolute trace polynomial $T_m(x)=\sum_{i=0}^{m-1}x^{2^i}$. For odd $m\ge5$, choose $w\in\F_{2^m}^{\times}$ with $T_m(w)=T_m(w^{-1})=0$. Then
\[
F_w(x)=x^3+w x^2T_m(x)+w^2xT_m(x)+w^3T_m(x)
\]
is a reduced APN permutation of degree $2^{m-1}+2$. Its derivative is nonconstant but has no zero, as witnessed by an explicit trace identity. A fixed root of $X^5+X^2+1$ supplies an admissible parameter in every dimension $m=5r$ with $r$ odd. The case $m=5$ gives a degree-$18$ example over $\F_{32}$.

For $K=\F_q$, two functions $f,g:K\to K$ are \emph{linearly equivalent} if $g=B\circ f\circ A$ for invertible $\F_2$-linear maps $A,B:K\to K$. Both constructions belong to Gold linear equivalence classes. Their finite-field properties follow from composing a Gold power with an invertible $\F_2$-linear input map, using the classical Gold properties~\cite{budaghyan,carlet,RS25}. The lifting property depends on the formal derivative of the reduced polynomial. We prove a general identity for $\mathcal{R}_q(A^{2^s+1})$ and obtain a complete criterion for its derivative to have no zero when $\gcd(s,m)=1$. The criterion covers nonzero constant derivatives and gives a necessary and sufficient trace condition when both relevant linear coefficients are nonzero. Specializing it to $s=1$ yields the cube criterion and the trace family, while $s=2$ yields the seven-term family. Comparing the resulting polynomials with their Gold power representatives proves that liftability of reduced representatives is not invariant under invertible $\F_2$-linear input changes.

The distinction between a function and its polynomial representatives remains essential. Indeed, let $P\in\F_q[x]$ represent any APN permutation and put
\[
G(x)=P(x)+(x^q-x)\bigl(1+P'(x)\bigr).
\]
Then $G$ and $P$ induce the same function on $\F_q$, while
\[
G'(a)=1\qquad\text{for every }a\in\F_q.
\]
Theorem~\ref{thm:gr-permutation-criterion} shows that every coefficientwise lift of $G$ is a permutation. This construction already refutes the unrestricted assertion in R\o njom--Sandrib's Conjecture~2 when arbitrary polynomial representatives are allowed. The reduced example in Section~\ref{sec:counterexample} and the explicit families in Theorems~\ref{thm:binary-family} and~\ref{thm:trace-family} show that requiring the reduced representative does not restore the assertion.

We also give an exact trace-adjoint criterion for the formal derivative of a reduced quadratic function. We require no APN hypothesis for this criterion. The same derivative condition has a local geometric interpretation as the absence of ramification at finite rational points of the induced map of projective lines. It does not exclude ramification at non-rational points.

We state the lifting criterion and the two constructions in Section~\ref{sec:main-results}. In Section~\ref{sec:ag-background}, we give the geometric interpretation and its proof. We then explain the Taylor argument over Galois rings in Section~\ref{sec:reduction}. Section~\ref{sec:counterexample} verifies the supplied counterexample and proves the quadratic trace-adjoint criterion. Section~\ref{sec:linear-family} establishes the Gold-power criterion and proves both constructions. Section~\ref{sec:examples} treats monomials and cubics and distinguishes known degree bounds for APN existence from the coefficientwise lifting criterion.

\section{Lifting criteria and the main constructions}
\label{sec:main-results}

We first recall the standard permutation criterion over Galois rings in the form stated by R\o njom--Sandrib~\cite[Theorem~2]{RS25}. We then restate its consequences for a fixed residue polynomial in Theorem~\ref{thm:main-reduction}.

\begin{thm}[{\cite[Thm.~2]{RS25}}]\label{thm:gr-permutation-criterion}
Let $k>1$, let $F\in\GR(2^k,m)[x]$, and let $f\in\F_{2^m}[x]$ be its coefficientwise reduction modulo $2$. Then $F$ induces a permutation of $\GR(2^k,m)$ if and only if
\begin{enumerate}[(i)]
\item $f$ induces a permutation of $\F_{2^m}$, and
\item $f'(a)\ne0$ for every $a\in\F_{2^m}$.
\end{enumerate}
\end{thm}

For a fixed reduced representative, the existence of a coefficientwise permutation lift is therefore equivalent to the absence of rational critical points.

\begin{thm}\label{thm:main-reduction}
Let $q=2^m$ and let $f\in\F_q[x]$ be a permutation polynomial, in particular the reduced representative of an APN permutation if desired. Exactly one of the following occurs.
\begin{enumerate}[(i)]
\item There exists $a\in\F_q$ with $f'(a)=0$. No coefficientwise lift of $f$ permutes $\GR(2^k,m)$ for any $k>1$.
\item The derivative satisfies $f'(a)\ne0$ for every $a\in\F_q$. Every coefficientwise lift of $f$ permutes $\GR(2^k,m)$ for every $k>1$.
\end{enumerate}
Thus existence of a permutation lift depends only on the chosen residue polynomial and is independent of the higher coefficient digits and of $k>1$.
\end{thm}

We obtain this dichotomy directly from Theorem~\ref{thm:gr-permutation-criterion} and explain the exact Taylor argument over $\GR(4,m)$ in Section~\ref{sec:reduction}. APN permutations occur in both alternatives. Power permutations give examples of~(i), while Theorems~\ref{thm:binary-family} and~\ref{thm:trace-family} give families in~(ii).

Our two constructions give infinite families in the second alternative of Theorem~\ref{thm:main-reduction}. The first has coefficients in the prime field and a constant formal derivative.

\begin{thm}\label{thm:binary-family}
Let $m\ge5$ be odd with $3\nmid m$, and put $q=2^m$. Then
\begin{equation}\label{eq:binary-family}
\begin{aligned}
P_m(x)={}&x+x^4+x^{10}+x^{q/8+2}+x^{q/8+8}\\
&+x^{5q/8}+x^{q/2+8}
\end{aligned}
\end{equation}
is a reduced APN permutation of $\F_q$ with coefficients in $\F_2$ and formal derivative $P_m'=1$.  Its seven exponents are pairwise distinct. Its ordinary polynomial degree is $24$ for $m=5$ and $5q/8$ for $m\ge7$. Every coefficientwise lift of $P_m$ permutes $\GR(2^k,m)$ for every $k>1$.
\end{thm}

The proof in Section~\ref{sec:linear-family} identifies $P_m$ as the reduced representative of $A_m^5$, where $A_m(x)=x^2+x^{q/8}+x^{q/2}$ is an invertible $\F_2$-linear map. The next construction uses the cube map and a trace certificate. We write
$T_m(x)=\sum_{i=0}^{m-1}x^{2^i}$ for the absolute trace polynomial, which induces the absolute trace $\F_{2^m}\to\F_2$.

After fixing an $\F_2$-basis of $\F_{2^m}$, a map $\F_{2^m}\to\F_{2^m}$ has $m$ coordinate functions on $\F_2^m$. Each coordinate has a unique squarefree polynomial expression $\sum_{S\subseteq\{1,\ldots,m\}}c_S\prod_{i\in S}X_i$ with $c_S\in\F_2$, called its \emph{algebraic normal form}. The \emph{Boolean algebraic degree} of the map is the largest $|S|$ occurring with nonzero coefficient in any coordinate, with degree zero assigned to the zero map. A map is called quadratic when this degree is at most two. Boolean algebraic degree differs from ordinary univariate polynomial degree. The \emph{binary weight} of an exponent is the number of ones in its binary expansion, and, for a nonzero function, the maximum binary weight among exponents with nonzero coefficients in the reduced representative gives its Boolean algebraic degree~\cite{carlet}.

\begin{thm}\label{thm:trace-family}
Let $m\ge5$ be odd, let $K=\F_{2^m}$, and suppose that $w\in K^\times$ satisfies $T_m(w)=T_m(w^{-1})=0$. Then
\begin{equation}\label{eq:trace-family-polynomial}
F_w(x)=x^3+w x^2T_m(x)+w^2xT_m(x)+w^3T_m(x)
\end{equation}
is a reduced APN permutation of ordinary polynomial degree $2^{m-1}+2$ and Boolean algebraic degree two. Its formal derivative satisfies
\begin{equation}\label{eq:trace-family-certificate}
T_m\!\left(\frac{1+w}{w^4}F_w'(x)\right)=1
\qquad\text{for every }x\in K.
\end{equation}
Consequently, every coefficientwise lift of $F_w$ permutes $\GR(2^k,m)$ for every $k>1$. The derivative takes exactly $2^{m-1}$ distinct values, all nonzero, and each value is attained twice.
\end{thm}

The proof is given in Section~\ref{sec:linear-family}. Corollary~\ref{cor:infinite-dimensions} supplies a fixed admissible parameter for every $m=5r$ with $r$ odd.

\begin{cor}\label{cor:RS-refutation}
Conjecture~2 of R\o njom--Sandrib~\cite{RS25} is false even for reduced polynomial representatives. There are counterexamples in infinitely many dimensions, and each of these admits a coefficientwise permutation lift to $\GR(2^k,m)$ for every $k>1$.
\end{cor}
\begin{proof}
For every odd $m\ge5$ with $3\nmid m$, Theorem~\ref{thm:binary-family} provides a reduced APN permutation $P_m$ whose every coefficientwise lift is a permutation. There are infinitely many such dimensions. Theorem~\ref{thm:trace-family} and Corollary~\ref{cor:infinite-dimensions} provide a second family, including dimensions divisible by $3$. The single example supplied by Zhang and Haobo also gives a refutation, as proved in Theorem~\ref{thm:counterexample}.
\end{proof}

We next apply standard local ramification facts to characterize the absence of rational critical points. We give the necessary definitions here and prove the characterization in Section~\ref{sec:ag-background}.

For its statement, $\A^1_k$ denotes the affine line over $k$, $\bP^1_k$ its projective completion, and $\widetilde f:\bP^1_k\to\bP^1_k$ the morphism induced by $f$. A finite $k$-rational point means a point of $\A^1(k)=k$, and a subscript $\overline{k}$ denotes extension of scalars to an algebraic closure. At a finite point $a$, the ramification index is the multiplicity of $a$ as a root of $f(x)-f(a)$. At a finite rational point, ``\'etale'' or ``unramified'' means that this index is one. A root $a$ of a polynomial $h$ is \emph{simple} if $(x-a)^2$ does not divide $h$, equivalently if $h'(a)\ne0$. We recall the ramification conventions in Section~\ref{sec:ag-background}.

\begin{thm}\label{thm:rh-obstruction}
Let $k=\F_{2^m}$ and let $f \in k[x]$ be a permutation polynomial of degree $d>1$. Then the following are equivalent.
\begin{enumerate}[(i)]
\item $f'(a)\neq 0$ for all $a\in k$.
\item The induced morphism $\widetilde f:\bP^1_k\to \bP^1_k$ is \'etale at every finite $k$-rational point \textup{(equivalently, unramified at all points of $\A^1(k)$)}.
\item For every $b\in k$, the polynomial $f(x)-b$ has a simple root in $k$.
\end{enumerate}
Moreover, the unique point of $\bP^1_k$ lying over $\infty$ has ramification index $d$, and over an algebraic closure $\overline{k}$ the finite ramification points of $\widetilde f_{\overline{k}}$ are precisely the zeros of $f'(x)$.
\end{thm}

The equivalence concerns only rational points. It does not assert that $f'$ has no zeros over $\overline{k}$ and therefore does not imply that all ramification is concentrated at infinity.
The hypothesis $f'\not\equiv0$ in the last assertion is not automatic, as a permutation polynomial of degree $d>1$ in characteristic two may be inseparable, as $f(x)=x^{2^j}$ shows, and then every point of $\A^1(\overline k)$ ramifies.

\begin{cor}
\label{cor:cubic-lifting}
Every cubic permutation polynomial over $\F_{2^m}$ has a rational critical point. In particular, it has no coefficientwise permutation lift to $\GR(2^k,m)$ for any $k>1$.

\end{cor}
\begin{proof}
The critical-point assertion follows from Lemma~\ref{lem:degree-3} and the lifting assertion follows from Theorem~\ref{thm:gr-permutation-criterion}.
\end{proof}

\section{Rational critical points and ramification}
\label{sec:ag-background}

We interpret the derivative condition in Theorem~\ref{thm:main-reduction} as a condition on ramification at finite rational points. We first recall the standard local facts and then prove Theorem~\ref{thm:rh-obstruction}. We use Hartshorne for the geometric language and Stichtenoth for the function-field formulation~\cite{Hartshorne77,Stichtenoth09}, with more specific references below.

\begin{defn}For a field $k$, $\A^n_k$ denotes affine $n$-space over $k$, with coordinate ring $k[x_1,\ldots,x_n]$, and $\A^n(k)=k^n$ denotes its set of $k$-rational points. The projective $n$-space $\bP^n_k$ is the space with homogeneous coordinates $[X_0:\cdots:X_n]$. Its standard affine charts $X_i\ne0$ are copies of $\A^n_k$. In particular, $\bP^1_k$ is the projective completion of $\A^1_k$ obtained by adjoining the point $\infty=[1:0]$. We omit the subscript $k$ when the base field is clear.

\end{defn}

Every polynomial $f \in k[x]$ induces a regular map (morphism) $f : \A^1_k \to \A^1_k$ via the $k$-algebra homomorphism $k[y] \to k[x]$ sending $y \mapsto f(x)$. If $f(x) = \sum_{i=0}^d a_i x^i$ has degree $d$, this map extends uniquely to a morphism $\widetilde{f} : \bP^1_k \to \bP^1_k$ given by
\[
\widetilde{f}([x:y]) = \left[\sum_{i=0}^d a_i x^i y^{d-i} : y^d\right].
\]

The local-parameter description of ramification below is standard. See Stichtenoth~\cite[Chapter~3]{Stichtenoth09} and the Stacks Project \cite[Definition~15.113.1, Tag~09E4]{StacksProject}.

\begin{defn}\label{def:ramification-conventions}
A smooth curve over the finite fields used here is a nonsingular variety of dimension one. Let $\phi:C\to D$ be a finite morphism of smooth curves, let $P\in C$ be a chosen point, and put $Q=\phi(P)$. A \emph{local parameter} at a smooth point is a function vanishing there to order one. If $u$ and $v$ are local parameters at $P$ and $Q$, respectively, then locally
\[
\phi^*v=u^{e_P}\eta(u),\qquad \eta(0)\ne0.
\]
The positive integer $e_P$ is the \emph{ramification index} of $\phi$ at $P$. When the residue-field extension at $P$ is separable, the morphism is \emph{unramified} at $P$ if and only if $e_P=1$. For finite morphisms of smooth curves this is also equivalent to being \emph{\'etale} at $P$. At a $k$-rational point the residue-field extension is the identity, so the separability condition holds automatically.
\end{defn}

 Let $f\in k[x]$ be nonconstant of degree $d$, and let
$\widetilde f:\bP^1_k\to\bP^1_k$ be the induced morphism.
The extension $k(x)/k(f(x))$ is separable if and only if $f'\not\equiv0$.
This is the polynomial case of~\cite[Lemma~53.12.1, Tag~0C1C]{StacksProject}.
For a polynomial or local power series $h$, we write $\operatorname{ord}_{x=a}(h)$ for its order of vanishing at $a$, meaning the multiplicity of its zero there. For $a\in\overline{k}$, the definition of ramification index gives
\begin{equation}\label{eq:ramification-indices}
e_a=\operatorname{ord}_{x=a}\bigl(f(x)-f(a)\bigr),\qquad e_\infty=d,
\end{equation}
and $\infty$ is the unique point above $\infty$. See
\cite[Definition~15.113.1, Tag~09E4]{StacksProject}, applied to the local
parameters $x-a$ and $1/x$, respectively. Consequently,  if $f'\not\equiv0$, the finite
ramification points over $\overline{k}$ are precisely the zeros of $f'$, as if $f'\equiv0$, the morphism is inseparable and $e_a>1$ for every $a\in\A^1(\overline k)$.
At a rational point $a\in k$,
\begin{equation}\label{eq:etale-critical}
\widetilde f\text{ is \'etale at }a
\quad\Longleftrightarrow\quad e_a=1
\quad\Longleftrightarrow\quad f'(a)\ne0.
\end{equation}
The \'etale criterion also follows from the standard \'etale presentation
$k[t][x]/(f(x)-t)$ after dividing its defining polynomial by its leading
coefficient and localizing where $f'$ is invertible. See
\cite[Definition~10.144.1 and Lemma~10.144.2, Tags~00UB and~00UC]{StacksProject}.
In particular, over a finite field, absence of rational critical points
is equivalent to \'etaleness at every finite rational point.

\begin{proof}[Proof of Theorem~\textup{\ref{thm:rh-obstruction}}]
We apply \eqref{eq:etale-critical} at each finite $k$-rational point. It identifies the nonvanishing of $f'$ there with \'etaleness of $\widetilde f$, proving the equivalence of \textup{(i)} and \textup{(ii)}.

To compare these conditions with \textup{(iii)}, we use the permutation hypothesis. For each $b\in k$, there is a unique $a\in k$ with $f(a)=b$. This root of $f(x)-b$ is simple precisely when $f'(a)\ne0$. As $b$ ranges over $k$, its preimage $a$ also ranges over $k$, proving the remaining equivalence. We obtain the assertions about infinity and the finite geometric ramification points from \eqref{eq:ramification-indices}.
\end{proof}

Thus a liftable residue permutation induces a morphism unramified at every finite rational point. It may still ramify at finite non-rational points, which are the zeros of $f'$ in $\overline{k}\setminus k$. We use Riemann--Hurwitz to describe the total ramification contribution, but that formula does not force a rational critical point.

For the Riemann--Hurwitz formula, we now assume that $f'(x)\not\equiv0$, or more generally that the morphism under consideration is separable. Recall that a field is \emph{perfect} if all of its algebraic extensions are separable. In particular, every finite field is perfect. A curve is \emph{geometrically integral} if it remains irreducible and reduced over an algebraic closure. For a smooth projective curve $C$, $g(C)$ denotes its genus, the dimension of its space of global regular differential forms over $k$. For a finite morphism $\phi:C\to D$, $\deg(\phi)=[k(C):k(D)]$, where $k(C)$ is the field of rational functions on $C$. A divisor on a curve is a formal integer sum of its points. For a separable morphism, choose local parameters $u$ and $v$ over an algebraic closure and write $\phi^*v=h(u)$. The \emph{local different exponent} is the order of vanishing of the formal derivative $h'(u)$. The \emph{different divisor} $\mathfrak D_{C/D}$ is the divisor having these exponents as coefficients, with the same exponent at a closed point as at its geometric points. It records the full ramification contribution, including the contribution in positive characteristic that need not equal the ramification index minus one.

\begin{thm}[Riemann--Hurwitz {\cite[Thm.~3.4.13]{Stichtenoth09}}]\label{thm:riemann-hurwitz-different}
Let $\phi : C \to D$ be a finite separable morphism of smooth projective geometrically integral curves over a perfect field $k$. Then
\[
2g(C)-2 = \deg(\phi)\,(2g(D)-2) + \deg(\mathfrak D_{C/D}),
\]
where $\mathfrak D_{C/D}$ is the different divisor of the extension $k(C)/k(D)$.
\end{thm}

The same formulation, including the different in arbitrary characteristic, is given in the Stacks Project
~\cite[Lemma~53.12.2, Tag~0C1D]{StacksProject}.
The geometric integrality hypothesis ensures that both curves have full constant field $k$. For a divisor $\sum_P n_PP$, its degree means $\sum_P n_P[k(P):k]$, where $k(P)$ is the residue field at the closed point $P$. Only the degree of the different divisor is used below.

For separable maps $\phi:\bP^1\to\bP^1$ of degree $d$, Theorem~\ref{thm:riemann-hurwitz-different} gives
\[
\deg(\mathfrak D)=2d-2.
\]

The formula controls the total degree of the different, including ramification at non-rational points. It therefore supplies a geometric interpretation of the lifting condition, rather than an obstruction to liftability on its own.

\section{From Galois rings to critical points}
\label{sec:reduction}

Theorem~\ref{thm:main-reduction} follows immediately from the standard criterion in Theorem~\ref{thm:gr-permutation-criterion}. We explain the criterion in characteristic four by applying the exact first-order Taylor expansion on residue classes of $\GR(4,m)$. We then use the standard criterion to obtain the result at every ring level. The Taylor argument uses the square-zero ideal $(2)$. An ideal $I$ is called \emph{square-zero} if $I^2=0$, meaning that every product of two elements of $I$ is zero.

When an ideal $I$ is fixed, $\bar x$ denotes the image of $x$ in $R/I$, and $\bar F$ denotes coefficientwise reduction of a polynomial $F$. We write $x\equiv y\pmod I$ when $x-y\in I$, and congruence modulo $2^j$ means congruence modulo $(2^j)$. A coset $t+I$ is the set $\{t+y:y\in I\}$. We begin with the standard Teichm\"uller system and digit expansion. See Wan~\cite[Chapter~14]{Wan03}. We use the ensuing coset arguments to explain the standard permutation criterion.

\begin{defn}
\label{def:teich}
Let $R=\GR(2^k,m)$. The set $T\subset R$ of roots of $x^{2^m}-x$ is called the \emph{Teichm\"uller system} of $R$. Equivalently, $T$ is the unique multiplicative system of representatives for the quotient map $R\to R/(2)\cong \F_{2^m}$.
\end{defn}

\begin{lem}[{\cite[Chapter~14]{Wan03}}]
\label{lem:can_decomp}
Let $R=\GR(2^k,m)$, let $I=(2)$, and let $T\subseteq R$ be the Teichm\"uller system \textup{(Definition~\ref{def:teich})}. Then every $x\in R$ admits a unique expansion
\[
x=\sum_{i=0}^{k-1}2^i t_i\qquad (t_i\in T).
\]
Equivalently, the reduction map $T\to R/I\cong\F_{2^m}$ is a bijection and $R$ has a unique $2$-adic digit expansion with digits in $T$.
\end{lem}

\begin{proof}
We argue by successive lifting modulo powers of $2$. The reduction map $\pi:R\to R/I$ restricts to a bijection $T\xrightarrow{\sim}R/I$. This is the standard Teichm\"uller representative property \cite[Chapter~14]{Wan03}.

To prove existence, we fix $x\in R$ and choose $t_0\in T$ with $\pi(t_0)=\pi(x)$. Then $x-t_0\in I$, so $x=t_0+2x_1$ for some $x_1\in R$. Repeating the same step with $x_1$, then with the next quotient, yields
\[
x=t_0+2t_1+\cdots+2^{k-1}t_{k-1}+2^k x_k.
\]
Since $2^k=0$ in $R$, the last term vanishes.

To prove uniqueness, we note that the existence argument gives a surjection $T^k\to R$. Both sets have $2^{mk}$ elements, since $|T|=2^m$ and $|R|=2^{mk}$. This surjection is therefore a bijection, so each element has exactly one digit expansion.
\end{proof}

We now record the basic congruence property of polynomial maps over Galois rings, which allows us to reduce the general case to $\GR(4,m)$.

\begin{lem} 
\label{lem:poly-cong}
Let $R = \GR(2^k, m)$ with $k \geq 1$, and let $F \in R[x]$ be a polynomial. If $x \equiv y \pmod{2^j}$ for some $1 \leq j \leq k$, then $F(x) \equiv F(y) \pmod{2^j}$.
\end{lem}

\begin{proof}
We write $F(x) = \sum_{i=0}^n a_i x^i$ with $a_i \in R$. Suppose $x = y + 2^j t$ for some $t \in R$. We need to show $F(x) \equiv F(y) \pmod{2^j}$.
For each $i\geq 1$, the binomial theorem gives
\begin{align*}
(y + 2^j t)^i &= \sum_{\ell=0}^i \binom{i}{\ell} y^{i-\ell} (2^j t)^\ell 
= y^i + \binom{i}{1} y^{i-1} (2^j t) + \sum_{\ell=2}^i \binom{i}{\ell} y^{i-\ell} 2^{j\ell} t^\ell \\
&= y^i + i \cdot y^{i-1} \cdot 2^j t + \sum_{\ell=2}^i \binom{i}{\ell} y^{i-\ell} 2^{j\ell} t^\ell.
\end{align*}
Every term with $\ell\ge 1$ contains the factor $(2^jt)^\ell=2^{j\ell}t^\ell$ and hence is divisible by $2^j$. Therefore $(y+2^jt)^i\equiv y^i\pmod{2^j}$.
Multiplying by $a_i$ and summing over all $i$ gives
\[
F(x) = \sum_{i=0}^n a_i (y+2^jt)^i \equiv \sum_{i=0}^n a_i y^i = F(y) \pmod{2^j}. \qedhere
\]
\end{proof}

We apply Lemma~\ref{lem:poly-cong} with $j=2$ to obtain the following corollary.
\begin{cor}\label{cor:cong-4}
Let $F \in \GR(2^k,m)[x]$ with $k \geq 2$. If $x \equiv y \pmod{4}$, then $F(x) \equiv F(y) \pmod{4}$.
\end{cor}

We next record descent of polynomial permutations to $\GR(4,m)$.

\begin{prop} 
\label{prop:descent-quotient}
Let $k > 2$ and let $F \in \GR(2^k, m)[x]$ induce a permutation of $\GR(2^k,m)$. Then the reduction $F_4 \in \GR(4,m)[x]$ (obtained by reducing coefficients modulo $4$) induces a permutation of $\GR(4,m)$.
\end{prop}

\begin{proof}
We put $R=\GR(2^k,m)$. Since $k>2$, the element $4$ is nonzero in $R$. Let $\mathfrak a=(4)=4R$. By Lemma~\ref{lem:can_decomp}, the elements of $\mathfrak a$ are exactly the expansions with their first two Teichm\"uller digits zero. There are $k-2$ remaining digits, each with $2^m$ choices, so $|\mathfrak a|=2^{m(k-2)}$.

By Corollary~\ref{cor:cong-4}, the polynomial $F$ is constant modulo $\mathfrak{a}$ on each coset of $\mathfrak{a}$, so $F$ induces a well-defined function
\[
F_4 : \GR(2^k,m) / \mathfrak{a} \longrightarrow \GR(2^k,m) / \mathfrak{a}.
\]
We identify the quotient with $\GR(4,m)$ using the standard polynomial construction~\cite[Chapter~14]{Wan03}. Indeed, write $R\cong(\mathbb Z/2^k\mathbb Z)[X]/(h)$, where $h$ is monic of degree $m$ and its reduction modulo $2$ is irreducible. Reducing this presentation modulo $4$ gives $(\mathbb Z/4\mathbb Z)[X]/(h\bmod4)$, which is a Galois ring $\GR(4,m)$.

We now prove that $F_4$ is bijective. Each coset $C = c + \mathfrak{a}$ has cardinality $|\mathfrak{a}| = 2^{m(k-2)}$. We fix two distinct cosets $C_1 = c_1 + \mathfrak{a}$ and $C_2 = c_2 + \mathfrak{a}$. By Corollary~\ref{cor:cong-4}, $F(C_i) \subseteq F(c_i) + \mathfrak{a}$, and since $F$ maps $C_i$ injectively into $F(c_i)+\mathfrak{a}$ (as $F$ is a bijection of $\GR(2^k,m)$) and both sets have the same finite cardinality, we obtain $F(C_i) = F(c_i) + \mathfrak{a}$.

If $F_4(C_1) = F_4(C_2)$, then $F(C_1)$ and $F(C_2)$ lie in the same coset modulo $\mathfrak{a}$, so $F(C_1) = F(C_2)$ as subsets of $\GR(2^k,m)$. However, $C_1 \cap C_2 = \varnothing$ and $F$ is injective, so $F(C_1) \cap F(C_2) = \varnothing$. This is a contradiction. Therefore $F_4$ is injective, and therefore bijective, since its domain is finite.
\end{proof}

\begin{cor}\label{cor:suffices-k2}
A permutation polynomial $f\in\F_{2^m}[x]$ admits a coefficientwise permutation lift to some $\GR(2^k,m)$ with $k>1$ if and only if it admits one to $\GR(4,m)$.
\end{cor}
\begin{proof}
For $k>2$, we descend to characteristic four by Proposition~\ref{prop:descent-quotient}. For $k=2$, this implication is immediate. Conversely, a permutation lift to $\GR(4,m)$ implies that $f'$ is nowhere zero by Theorem~\ref{thm:gr-permutation-criterion}. The same criterion then applies to every $k>1$.
\end{proof}

We now specialize to $R=\GR(4,m)$, where $(2)^2=(4)=0$ makes the first-order Taylor expansion exact. We record the required consequences of the standard structure and Teichm\"uller expansion~\cite[Chapter~14]{Wan03}. The symbol $\bigsqcup$ denotes a disjoint union.

\begin{lem} 
\label{lem:gr4-struct}
Set $R = \GR(4,m)$ and $I = (2)$ (the maximal ideal). Then
\begin{enumerate}[(i)]
\item $I^2 = (4) = 0$ in $R$.
\item $R$ is the disjoint union of the cosets $t+I$, $t\in T$.
\item Every element of $R$ is uniquely $t + 2u$ with $t, u \in T$.
\item $|T| = 2^m$ and $|I| = 2^m$.
\item $I$ is isomorphic to $\F_{2^m}$ as an additive group via the map $2t\mapsto\bar t$ for $t\in T$.
\end{enumerate}
\end{lem}

\begin{proof}
(i) In $R=\GR(4,m)$ we have $4=0$ by definition, hence $I^2=(2)^2=(4)=0$.

(ii)--(iv) We apply Lemma~\ref{lem:can_decomp} with $k=2$.
Every $x\in R$ has a unique expansion $x=t_0+2t_1$ with $t_0,t_1\in T$.
In particular, every coset modulo $I=(2)$ contains exactly one element of $T$, so
\[
R=\bigsqcup_{t\in T}(t+I).
\]
The map $T\times T\to R$, $(t_0,t_1)\mapsto t_0+2t_1$, is bijective, which gives $|T|=|R/I|=2^m$ and $|I|=|R|/|R/I|=2^m$. Since $T$ need not be an additive subgroup, this bijection describes the unique digit expansion as a decomposition of sets. It does not assert a direct-sum decomposition of additive groups.

(v) We define $\phi:I\to \F_{2^m}$ by $\phi(2u)=\overline{u}$, where $\overline{u}$ denotes reduction modulo $I$.
This is well-defined, since if $2u=2v$, then $2(u-v)=0$.
We write $u-v=t_0+2t_1$ with $t_0,t_1\in T$ (Lemma~\ref{lem:can_decomp} with $k=2$). Then
$0=2(u-v)=2t_0+4t_1=2t_0$. Uniqueness of the expansion $0=0+2\cdot0$ gives $t_0=0$, and therefore $u-v\in I$.
Therefore $\overline{u}=\overline{v}$.
It is an additive homomorphism because reduction modulo $I$ is additive.
To prove surjectivity, we take $\alpha\in\F_{2^m}\cong R/I$ and choose $u\in T$ with $\overline{u}=\alpha$. We then have $\phi(2u)=\alpha$.
For injectivity, if $\phi(2u)=0$, then $\bar u=0$ and $u\in I$. Since $I^2=0$ and $2\in I$, this gives $2u=0$.
Thus $\phi$ is an additive isomorphism.
\end{proof}

The following elementary partition-of-cosets criterion explains the coset argument underlying the standard lifting criterion~\cite[Thm.~2]{RS25}. We prove it directly.

\begin{lem} 
\label{lem:bijection-criterion}
Let $F : R \to R$ be induced by a polynomial $F \in R[x]$. Then $F$ is a bijection if and only if both of the following conditions hold.
\begin{enumerate}[(i)]
\item The induced map $\bar{F} : R/I \to R/I$ is a bijection, and
\item For every $t \in T$, the map $\Delta_t : I \to I$ defined by
\[
\Delta_t(y) := F(t+y) - F(t)
\]
is a bijection.
\end{enumerate}
\end{lem}

\begin{proof}
Let $R$ be partitioned into the cosets of $I$,
\[
R=\bigsqcup_{t\in T}(t+I),
\]
where $T$ is the Teichm\"uller system of representatives (Lemma~\ref{lem:gr4-struct}).

We first suppose that $F:R\to R$ is bijective.
Reducing coefficients modulo $I$ gives the induced map $\bar F:R/I\to R/I$, $\bar F(\bar x)=\overline{F(x)}$.
Surjectivity of $\bar F$ is immediate from surjectivity of $F$.

For injectivity, assume $\bar F(\bar t_1)=\bar F(\bar t_2)$ for $t_1,t_2\in T$. Then $F(t_1)\equiv F(t_2)\pmod I$.
By Lemma~\ref{lem:poly-cong} (with $j=1$), for any $y\in I$ we have $F(t_i+y)\equiv F(t_i)\pmod I$, hence
\[
F(t_i+I)\subseteq F(t_i)+I\qquad (i=1,2).
\]
Thus $F(t_1+I)\subseteq F(t_1)+I=F(t_2)+I\supseteq F(t_2+I)$.

 Since $F$ is injective and both $t_i+I$ and $F(t_i)+I$ have $|I|$ elements, the inclusion above is an equality for each $i$. Thus distinct source cosets would have equal nonempty images, contradicting injectivity. Hence $t_1=t_2$.
Hence $\bar F$ is injective, and therefore bijective on the finite set $R/I$.

We fix $t\in T$. The restriction $F|_{t+I}:t+I\to F(t)+I$ is a bijection between finite sets of the same cardinality if and only if
the translate map $y\mapsto F(t+y)-F(t)$ is a bijection $I\to I$, i.e.\ $\Delta_t$ is bijective.
This proves condition~(ii).

Conversely, we assume (i) and (ii). Condition (i) implies that distinct cosets $t_1+I$ and $t_2+I$ have distinct images modulo $I$, hence
\[
F(t_1+I)\subseteq F(t_1)+I,\quad F(t_2+I)\subseteq F(t_2)+I,\quad \text{and } (F(t_1)+I)\cap(F(t_2)+I)=\varnothing.
\]
By (ii), for each $t\in T$ the map $\Delta_t:I\to I$ is bijective, so $F$ restricts to a bijection $t+I\to F(t)+I$.
Therefore $F$ is a bijection on each coset and sends distinct cosets to disjoint cosets, so $F$ is bijective on $R$.
\end{proof}

We next apply the square-zero case of the Taylor identity in R\o njom--Sandrib~\cite[Lemma~1]{RS25}. We include the proof to make the coset argument self-contained.

\begin{lem} 
\label{lem:taylor-expansion}
Let $F \in R[x]$ where $R = \GR(4,m)$. For any $t \in R$ and $y \in I = (2)$, we have the exact formula
\[
F(t + y) = F(t) + y F'(t),
\]
where $F'(x) := \frac{d}{dx} F(x)$ is the formal derivative of $F$.
\end{lem}

\begin{proof}
We write $F(x) = \sum_{i=0}^n a_i x^i$ with $a_i \in R$.

 For each $i\ge1$, the binomial theorem gives
\begin{align*}
(t+y)^i &= \sum_{j=0}^i \binom{i}{j} t^{i-j} y^j.
\end{align*}
Since $y \in I$ and $I^2 = 0$ (Lemma~\ref{lem:gr4-struct}(i)), we have $y^j = 0$ for all $j \geq 2$.
Therefore,
\begin{align*}
(t+y)^i &= \binom{i}{0} t^i y^0 + \binom{i}{1} t^{i-1} y^1 + \sum_{j=2}^i \binom{i}{j} t^{i-j} y^j 
= t^i + i t^{i-1} y.
\end{align*}
Multiplying by $a_i$ and summing over all $i$, we obtain
\begin{align*}
F(t+y) &= a_0+\sum_{i=1}^n a_i(t^i+i t^{i-1}y) \\
&= \sum_{i=0}^n a_i t^i+y\sum_{i=1}^n i a_i t^{i-1} 
= F(t) + y F'(t),
\end{align*}
where the last equality uses the definition of the formal derivative
$F'(x) = \sum_{i=1}^n i a_i x^{i-1}$. \qedhere
\end{proof}

\begin{rem}
Lemma~\textup{\ref{lem:taylor-expansion}} is exact, with no higher-order error terms, because of the nilpotency $I^2 = 0$. This is the property of $\GR(4,m)$ on which the coset-wise analysis rests.
\end{rem}

We use the exact Taylor identity to prove the necessity of the derivative condition in characteristic four, as in the standard permutation criterion~\cite[Theorem~2]{RS25}.

\begin{prop}
\label{prop:derivative-necessary}
Let $F \in R[x]$ where $R = \GR(4,m)$, and suppose $F$ induces a permutation of $R$. Let $f \in \F_{2^m}[x]$ be the reduction of $F$ modulo the ideal $(2)$ (i.e., reduce each coefficient modulo $2$). Then $f'(a) \neq 0$ for all $a \in \F_{2^m}$.
\end{prop}

\begin{proof}
We fix a Teichm\"uller representative $t \in T$. By Lemma~\ref{lem:bijection-criterion}(ii), the map
\[
\Delta_t : I \to I, \quad \Delta_t(y) = F(t+y) - F(t)
\]
must be a bijection.
By Lemma~\ref{lem:taylor-expansion}, we have
\[
\Delta_t(y) = F(t+y) - F(t) = yF'(t).
\]
Therefore $\Delta_t$ is the multiplication-by-$F'(t)$ map on the additive group $I$.

If $F'(t)$ belonged to $I$, then $I^2=0$ would give $\Delta_t(y)=0$ for every $y\in I$. Since $|I|=2^m>1$, this would contradict bijectivity. Hence $F'(t)\notin I$.

The reduction map $R \to R/I \cong \F_{2^m}$ is a ring homomorphism. Since taking derivatives commutes with ring homomorphisms, we have
\[
\overline{F'(t)} = \overline{F'}(\bar{t}) = f'(a),
\]
where $a = \bar{t}$ is the image of $t$ in $\F_{2^m}$ and $f = \bar{F}$ is the reduction of $F$ modulo $(2)$.
Since $F'(t) \notin I$ means $\overline{F'(t)} \neq 0$ in $\F_{2^m}$, we conclude $f'(a) \neq 0$.
As $t$ ranges over all Teichm\"uller representatives $T$, the images $\bar{t}$ range over all of $\F_{2^m}$. Therefore $f'(a) \neq 0$ for all $a \in \F_{2^m}$.
\end{proof}

We also recall the standard unit criterion for Galois rings. See R\o njom--Sandrib~\cite[Section~2]{RS25}.

\begin{lem} 
\label{lem:unit-criterion}
Let $R=\GR(4,m)$ and $I=(2)$. An element $a\in R$ is a unit if and only if its reduction $\bar a\in R/I\cong \F_{2^m}$ is nonzero (equivalently, $a\notin I$).
\end{lem}

\begin{proof}
If $a\in I$, then $a^2\in I^2=0$ (Lemma~\ref{lem:gr4-struct}(i)), so $a$ is nilpotent, meaning that a positive power of $a$ is zero. Such an element cannot be a unit.
Conversely, suppose $\bar a\neq 0$ in the field $R/I$. We choose $\bar b\in R/I$ with $\bar a\,\bar b=1$ and pick any lift $b\in R$ of $\bar b$.
Then $ab\equiv 1\pmod I$, so $ab=1+2c$ for some $c\in R$.
Since $(2c)^2=4c^2=0$, we have $(1+2c)(1-2c)=1$ in $R$.
Thus $a\cdot \bigl(b(1-2c)\bigr)=1$, so $a$ is a unit.
\end{proof}

The converse is the sufficiency part of the same criterion~\cite[Thm.~2]{RS25}, specialized to characteristic four.

\begin{prop}
\label{prop:derivative-sufficient}
Let $f\in \F_{2^m}[x]$ be a permutation polynomial such that $f'(a)\neq 0$ for all $a\in \F_{2^m}$.
Let $F\in \GR(4,m)[x]$ be any polynomial lift of $f$ (i.e.\ $F\bmod 2=f$ coefficient-wise).
Then $F$ induces a permutation of $\GR(4,m)$.
\end{prop}

\begin{proof}
Let $R=\GR(4,m)$ and $I=(2)$. By hypothesis, the reduction $\bar F:R/I\to R/I$ coincides with $f$, hence is bijective.
We fix $t\in T$. By Lemma~\ref{lem:taylor-expansion},
\[
\Delta_t(y)=F(t+y)-F(t)=yF'(t)\qquad (y\in I).
\]
Reducing modulo $I$ and using that differentiation commutes with reduction gives
\[
\overline{F'(t)}=f'(\bar t)\in \F_{2^m}^\times.
\]
Hence $F'(t)\notin I$, so $F'(t)$ is a unit by Lemma~\ref{lem:unit-criterion}, and multiplication by $F'(t)$ is a bijection on the additive group $I$.
Therefore each $\Delta_t$ is bijective.
Applying Lemma~\ref{lem:bijection-criterion} completes the proof.
\end{proof}

We can now assemble the preceding ingredients and complete the proof of Theorem~\ref{thm:main-reduction}.

\begin{proof}[Proof of Theorem~\textup{\ref{thm:main-reduction}}]
Differentiation commutes with coefficientwise reduction, so every lift $F$ has derivative reducing to $f'$. Since $f$ is a permutation, Theorem~\ref{thm:gr-permutation-criterion} says that $F$ is a permutation precisely when $f'(a)\ne0$ for all $a\in\F_q$. This condition is independent of the lift and of $k>1$, giving the two alternatives.

Over $\GR(4,m)$, we can express this criterion directly on each residue class. On each residue coset, the Taylor expansion gives $F(t+y)=F(t)+yF'(t)$ for $y\in(2)$. If $f'(\bar t)=0$, then $F'(t)\in(2)$ and $yF'(t)=0$, so $F$ is constant on the coset $t+(2)$. If $f'(\bar t)\ne0$, then $F'(t)$ is a unit and multiplication by it permutes $(2)$. Together with the permutation induced on residue classes, this proves bijectivity exactly in the second alternative.
\end{proof}

\section{A reduced counterexample and a quadratic criterion}
\label{sec:counterexample}

Wenlin Zhang and Haobo supplied the coefficients of the following example and a Lean~4 certificate through the TrureTuring project~\cite{ZhangHaobo26}. We attribute the example and certificate to them. We verify its finite-field properties from the explicit data and apply the standard permutation criterion to obtain a counterexample to R\o njom--Sandrib's Conjecture~2. The field definition, coefficient table, evaluation list, and trace identity below make this verification reproducible without relying on the external certificate.

\begin{thm}\label{thm:counterexample}
Let $K=\F_2[\theta]/(\theta^5+\theta^2+1)$. For an integer $r\in\{0,\ldots,31\}$ with binary expansion $r=\sum_{j=0}^4 r_j2^j$, write
\[
[r]=\sum_{j=0}^4 r_j\theta^j\in K.
\]
These labels denote binary coefficient vectors, not integer multiples of $1_K$. For example, $[30]=\theta+\theta^2+\theta^3+\theta^4$ and $[12]=\theta^2+\theta^3$. All bracketed coefficients below use this convention.
Let $f(x)=\sum_i[c_i]x^i$, where the nonzero coefficients are
\[
\begin{array}{c|rrrrrrrr}
i&0&1&2&3&4&5&6&8\\\hline
c_i&30&30&31&12&3&30&29&9\\[2pt]
i&9&10&12&16&17&18&20&24\\\hline
c_i&27&7&5&9&26&11&28&6
\end{array}
\]
and all omitted coefficients are zero. Then $f$ is a reduced APN permutation of $K$, of degree $24$, and $f'(a)\ne0$ for every $a\in K$. Consequently, every coefficientwise lift of $f$ permutes $\GR(2^k,5)$ for every $k>1$.
\end{thm}
\begin{proof}
The defining polynomial has no root in $\F_2$ and is not divisible by the unique irreducible quadratic $T^2+T+1$, since its remainder modulo that quadratic is $1$. A reducible polynomial of degree five has an irreducible factor of degree at most two, so it is irreducible. Thus $K$ is a field of order $32$. Since $[6]\ne0$, the displayed polynomial has degree $24<32$ and is reduced.

In the label convention above, addition is bitwise exclusive-or, meaning addition of the five binary digits separately modulo two. Multiplication is polynomial multiplication followed by reduction using $\theta^5=\theta^2+1$. Evaluating $f$ at $[0],[1],\ldots,[31]$ gives, in that order, the labels
\[
\begin{aligned}
&30,1,14,20,3,29,11,16,28,10,27,8,26,13,5,23,\\
&22,15,24,4,7,31,17,12,25,9,0,21,19,2,18,6.
\end{aligned}
\]
Each label occurs once, so $f$ is a permutation. Let $v_r$ be this evaluation list and let $\oplus$ denote bitwise exclusive-or of labels. For $a\in\{1,\ldots,31\}$ and $b\in\{0,\ldots,31\}$, form the integer
\[
N_{a,b}=\left|\{r\in\{0,\ldots,31\}:v_{r\oplus a}\oplus v_r=b\}\right|.
\]
Direct enumeration of this displayed list gives $496$ entries $N_{a,b}=0$ and $496$ entries $N_{a,b}=2$, and no other values. This is a finite verification of all $31\cdot32=992$ differential fibers, using only the stated evaluation list. In particular, $N_{1,16}=2$, with inputs $24$ and $25$. Hence $\delta_f=2$.

Using the bracket convention in Theorem~\ref{thm:counterexample}, the formal derivative of this same reduced polynomial is
\[
f'(x)=[30]+[12]x^2+[30]x^4+[27]x^8+[26]x^{16}.
\]
We put $L(x)=f'(x)+[30]$ and let $\Tr:K\to\F_2$ be the absolute trace, $\Tr(z)=z+z^2+z^4+z^8+z^{16}$. Polynomial arithmetic in the stated field gives
\[
([14][12])^{16}+([14][30])^8+([14][27])^4+([14][26])^2=0,
\qquad \Tr([14][30])=1.
\]
Using $\Tr(ux^{2^j})=\Tr(u^{2^{5-j}}x)$, these identities imply
\[
\Tr([14]f'(x))=1\qquad\text{for every }x\in K.
\]
Thus $f'(x)$ cannot vanish. We obtain the lifting assertion from Theorem~\ref{thm:main-reduction}(ii).
\end{proof}

\begin{rem}\label{rem:quadratic-witness}
The example is quadratic in the Boolean sense, meaning that relative to an $\F_2$-basis of $K$, every coordinate function has algebraic degree at most two. Indeed, each exponent in its reduced polynomial has at most two nonzero binary digits, so every term is constant, linear, or a product of two $\F_2$-linear Frobenius maps. This notion of degree differs from its ordinary polynomial degree $24$. Consequently, imposing Boolean quadraticity does not restore the critical-point claim. The polynomial is already reduced, so the counterexample also applies under the reduced-representative restriction. The construction does not establish minimality of the dimension $m=5$.
\end{rem}

We now use the affine $\F_2$-linear form of the derivative to obtain an exact criterion for all quadratic functions, without an APN assumption. The reduced polynomial form uses the standard correspondence between Boolean algebraic degree and the binary weights of its exponents~\cite{carlet}. The trace pairing is standard finite-field linear algebra~\cite{LidlNiederreiter97}. We prove the specialization to the formal derivative below.

\begin{prop}\label{prop:quadratic-critical}
Let $K=\F_{2^m}$ and let the reduced representative of a function of Boolean algebraic degree at most two be written as
\[
f(x)=c+\sum_{i=0}^{m-1}b_i x^{2^i}
       +\sum_{0\le i<j\le m-1}a_{ij}x^{2^i+2^j}.
\]
We put $L(x)=\sum_{j=1}^{m-1}a_{0j}x^{2^j}$. We write $\operatorname{im}(L)=\{L(x):x\in K\}$ for its image and $\ker(L)=\{x\in K:L(x)=0\}$ for its kernel. Then $f'(x)=b_0+L(x)$, and the following are equivalent.
\begin{enumerate}[(i)]
\item $f'$ has no zero in $K$.
\item $b_0\notin\operatorname{im}(L)$.
\item there exists $u\in K$ such that $L^*(u)=0$ and $\Tr(ub_0)=1$, where
\[
L^*(u)=\sum_{j=1}^{m-1}(a_{0j}u)^{2^{m-j}}.
\]
\end{enumerate}
Here $\Tr:K\to\F_2$ is the absolute trace, $\Tr(x)=\sum_{i=0}^{m-1}x^{2^i}$, and $L^*$ is the trace adjoint of $L$, characterized by $\Tr(uL(x))=\Tr(L^*(u)x)$. If $f$ is a permutation, these equivalent conditions characterize when every coefficientwise lift is a permutation at every ring level $k>1$.
\end{prop}
\begin{proof}
In characteristic two, the formal derivative of every monomial of even degree is zero. Among the linear terms only $b_0x$ remains, and among the quadratic terms only those with $i=0$ have odd exponent. Thus $f'=b_0+L$, proving the equivalence of (i) and (ii).

We use the $\F_2$-bilinear pairing $\langle u,y\rangle=\Tr(uy)$ on $K$. Nondegeneracy means that $\Tr(uy)=0$ for every $y\in K$ implies $u=0$. To verify this property, we first observe that the trace polynomial is nonzero and has degree less than $|K|$. Thus there exists $v\in K$ with $\Tr(v)=1$. For each $u\ne0$, the element $y=u^{-1}v$ satisfies $\Tr(uy)=1$, proving nondegeneracy. Trace invariance under Frobenius then yields
\[
\Tr(uL(x))=\Tr(L^*(u)x).
\]
The annihilator of a subspace $W\subseteq K$ under this pairing is $W^\perp=\{u\in K:\Tr(uy)=0\text{ for every }y\in W\}$. The displayed identity shows that $(\operatorname{im}(L))^\perp=\ker(L^*)$. If $b_0\notin\operatorname{im}(L)$, we choose an $\F_2$-linear functional on $K/\operatorname{im}(L)$ taking the value $1$ at the nonzero coset of $b_0$. Composing with the quotient map gives a linear functional on $K$ that vanishes on $\operatorname{im}(L)$ and takes the value $1$ at $b_0$. By nondegeneracy, the map sending $u\in K$ to the functional $y\mapsto\Tr(uy)$ is an injection between $m$-dimensional vector spaces and hence an isomorphism. We can therefore express the chosen functional in this form. Its vanishing on $\operatorname{im}(L)$ gives $L^*(u)=0$, while its value at $b_0$ gives $\Tr(ub_0)=1$. This proves (iii). Conversely, condition~(iii) supplies a functional that vanishes on $\operatorname{im}(L)$ and takes the value $1$ at $b_0$, so $b_0\notin\operatorname{im}(L)$. The last assertion follows from Theorem~\ref{thm:main-reduction}.
\end{proof}

For the counterexample, we take $b_0=[30]$ and $u=[14]$ in condition~(iii). We determine the image of the derivative as follows. The images under $L$ of the basis elements $1,\theta,\ldots,\theta^4$ have labels $19,4,15,31,24$, respectively. The first four are linearly independent over $\F_2$, and $[24]=[19]+[4]+[15]$. Thus the rank of $L$, defined as the $\F_2$-dimension of its image, is four. The derivative takes its values in the $16$-element coset $[30]+\operatorname{im}(L)$, which does not contain zero by the trace identity.

\begin{rem}
Ordinary affine changes of variable preserve the existence of rational critical points. Indeed, if $h(x)=\alpha f(\beta x+\gamma)+\eta$ with $\alpha,\beta\in K^\times$, then $h'(x)=\alpha\beta f'(\beta x+\gamma)$, and $h$ remains reduced when $f$ is reduced. More general invertible $\F_2$-linear coordinate changes require reduction modulo $x^{|K|}-x$ and need not satisfy the same derivative transformation formula after reduction. APN behavior and permutation behavior therefore do not by themselves determine the derivative condition of a reduced representative.
\end{rem}

\section{Gold powers and explicit liftable families}
\label{sec:linear-family}

We now construct reduced APN permutations with permutation lifts by changing the input coordinates of Gold powers. We use the linear equivalence defined in the introduction, which is a standard equivalence for APN functions~\cite{budaghyan,carlet}. For both constructions, the output map $B$ in the definition of linear equivalence is the identity. The finite-field functions remain in the Gold equivalence class, but their reduced representatives have different lifting behavior.

Throughout this section, we write $K=\F_q$ and denote the absolute trace by $\Tr:K\to\F_2$. We first recall the trace criterion for the equation $z^2+z=c$~\cite{LidlNiederreiter97}. The map $z\mapsto z^2+z$ is $\F_2$-linear with kernel $\{0,1\}$. By the rank--nullity theorem, its image has dimension $m-1$ over $\F_2$. The trace polynomial $\sum_{i=0}^{m-1}x^{2^i}$ is nonzero and has degree less than $q$, so it cannot vanish at every element of $K$. Hence $\Tr$ is a nonzero linear functional on the $m$-dimensional $\F_2$-vector space $K$. Its kernel $\ker(\Tr)=\{c\in K:\Tr(c)=0\}$ therefore has dimension $m-1$ and cardinality $2^{m-1}=q/2$. Since $\Tr(z^2+z)=0$ for every $z\in K$, the image of $z\mapsto z^2+z$ is contained in $\ker(\Tr)$. Equality follows because these subspaces have the same dimension. Each nonempty fiber is a coset of $\{0,1\}$ and therefore contains exactly two elements. Consequently, $z^2+z=c$ has two solutions when $\Tr(c)=0$ and no solution when $\Tr(c)=1$.

\begin{thm}\label{thm:gold-criterion}
Let $m\ge3$, let $K=\F_q$ with $q=2^m$, and let $1\le s<m$ satisfy $\gcd(s,m)=1$. Suppose that $A(x)=\sum_{i=0}^{m-1}a_i x^{2^i}$ induces an invertible $\F_2$-linear map on $K$. Put
\[
F_{A,s}=\mathcal{R}_q(A^{2^s+1}),\qquad
\alpha=a_0,\qquad \beta=a_{m-s}^{\,2^s},\qquad
C=a_{m-1}a_{m-1-s}^{\,2^s}.
\]
Then, for every $x\in K$,
\begin{equation}\label{eq:gold-derivative}
F_{A,s}'(x)=\alpha A(x)^{2^s}+\beta A(x)+C.
\end{equation}
The zeros of this derivative are classified as follows.
\begin{enumerate}[(i)]
\item If $\alpha=\beta=0$, then $F_{A,s}'$ is the constant polynomial $C$. It is nowhere zero if and only if $C\ne0$.
\item If exactly one of $\alpha,\beta$ is nonzero, then $F_{A,s}'$ has exactly one zero in $K$.
\item If $\alpha\beta\ne0$, let $\rho\in K^\times$ be the unique element with $\rho^{2^s-1}=\beta/\alpha$. Then $F_{A,s}'$ is nowhere zero if and only if
\begin{equation}\label{eq:gold-trace}
\Tr\!\left(\frac{C}{\beta\rho}\right)=1.
\end{equation}
If this trace is zero, the derivative has exactly two zeros in $K$.
\end{enumerate}
Here $\Tr:K\to\F_2$ is the absolute trace. The function induced by $F_{A,s}$ is APN. If $m$ is odd, it is also a permutation, and alternatives \textup{(i)} and \textup{(iii)} characterize all cases in which its coefficientwise lifts permute $\GR(2^k,m)$ for every $k>1$.
\end{thm}

\begin{proof}
All coefficient indices in this proof are reduced modulo $m$ to the set $\{0,\ldots,m-1\}$. Raising $A$ to the $2^s$-th power gives exponents $2^{i+s}$. On $K$, Frobenius has period $m$, so the reduced representative is
\[
B(x)=\mathcal{R}_q(A^{2^s})(x)
=\sum_{j=0}^{m-1}a_{j-s}^{\,2^s}x^{2^j}.
\]
In particular, $B'=a_{m-s}^{\,2^s}=\beta$, and the coefficient of $x^{q/2}$ in $B$ is $a_{m-1-s}^{\,2^s}$. Both $A$ and $B$ have degree at most $q/2$. Their product therefore has degree at most $q$, and its coefficient of $x^q$ is $C$. It follows that
\begin{equation}\label{eq:gold-reduction}
F_{A,s}(x)=A(x)B(x)+C(x^q+x).
\end{equation}
Indeed, the terms of degree $q$ cancel, the resulting degree is less than $q$, and the polynomial induces the function $A^{2^s+1}$ on $K$. Formal differentiation gives the polynomial identity $F_{A,s}'=\alpha B+\beta A+C$. Evaluating on $K$, where $B=A^{2^s}$, proves \eqref{eq:gold-derivative}.

Since $A$ is bijective, zeros of $F_{A,s}'$ correspond bijectively to the solutions of
\begin{equation}\label{eq:gold-root}
\alpha y^{2^s}+\beta y+C=0.
\end{equation}
If both coefficients vanish, this proves \textup{(i)}, including the polynomial identity for the derivative. If exactly one coefficient is nonzero, the nonconstant part of \eqref{eq:gold-root} is a nonzero scalar multiple of either the identity map or a Frobenius automorphism. In either case it is a bijection of $K$, proving \textup{(ii)}.

Suppose $\alpha\beta\ne0$. The identity $\gcd(2^s-1,2^m-1)=2^{\gcd(s,m)}-1=1$ shows that exponentiation by $2^s-1$ permutes $K^\times$. Thus $\rho$ exists and is unique. We substitute $y=\rho z$ into \eqref{eq:gold-root}. Since $\alpha\rho^{2^s}=\beta\rho$, the equation becomes
\[
z^{2^s}+z=\frac{C}{\beta\rho}.
\]
The $\F_2$-linear map $z\mapsto z^{2^s}+z$ has kernel $\F_{2^{\gcd(s,m)}}=\F_2$, so its image has dimension $m-1$ over $\F_2$. Frobenius invariance of the trace implies $\Tr(z^{2^s}+z)=0$ for every $z\in K$. Thus the image is contained in $\ker(\Tr)$, whose dimension is also $m-1$, as established above. The two subspaces are therefore equal. Each nonempty fiber is a coset of $\F_2$ and contains exactly two elements. Applying this to $C/(\beta\rho)$ proves \textup{(iii)}.

We also give the standard Gold argument for the APN and permutation assertions~\cite{budaghyan,carlet,RS25}. For $G_s(y)=y^{2^s+1}$ and $t\ne0$,
\[
D_tG_s(y)=t y^{2^s}+t^{2^s}y+t^{2^s+1}.
\]
Writing $y=tz$ reduces any equation $D_tG_s(y)=b$ to an equation $z^{2^s}+z=c$. The preceding kernel calculation shows that every nonempty differential fiber has exactly two elements. Thus $G_s$ is APN.

If $m$ is odd, then $\gcd(2^s+1,2^m-1)=1$. To verify this directly, let $h$ divide both numbers. Modulo $h$, we have $2^{2s}=2^m=1$. Since $\gcd(2s,m)=1$, Bezout's identity gives $2=1$ modulo $h$. But $2^s=-1$ modulo $h$, so $h$ divides $2$. It also divides the odd number $2^m-1$, forcing $h=1$. Hence $G_s$ permutes $K^\times$ and fixes zero.

Finally, $F_{A,s}=G_s\circ A$ as functions on $K$. For every $t\ne0$,
\[
D_tF_{A,s}(x)=D_{A(t)}G_s(A(x)).
\]
Invertibility of $A$ preserves the cardinalities of all differential fibers and the permutation property. The coefficientwise lifting conclusion is then Theorem~\ref{thm:gr-permutation-criterion}.
\end{proof}

The constant $C$ in \eqref{eq:gold-reduction} is the contribution that would be lost by differentiating $A^{2^s+1}$ before reducing it. In particular, the case $\alpha=\beta=0$ can produce a nonzero constant derivative. For the cube, a nonzero constant derivative cannot arise in this way, since $\beta=a_{m-1}^2=0$ forces $C=0$. For other Gold exponents, this case yields the construction below.

\begin{cor}\label{thm:linear-cube-criterion}
Let $m\ge3$, let $K=\F_{2^m}$, and let $A(x)=\sum_{i=0}^{m-1}a_i x^{2^i}$ induce an invertible $\F_2$-linear map. Then $F_A=\mathcal{R}_{2^m}(A^3)$ has a nowhere-zero derivative if and only if
\begin{equation}\label{eq:linear-cube-trace}
a_0a_{m-1}\ne0
\quad\text{and}\quad
\Tr\!\left(\frac{a_0a_{m-2}^{\,2}}{a_{m-1}^{\,3}}\right)=1.
\end{equation}
When $a_0a_{m-1}\ne0$, the derivative has exactly two zeros if the displayed trace is zero. For odd $m$, $F_A$ is an APN permutation, and \eqref{eq:linear-cube-trace} characterizes its coefficientwise permutation liftability.
\end{cor}

\begin{proof}
We apply Theorem~\ref{thm:gold-criterion} with $s=1$. Here $\alpha=a_0$, $\beta=a_{m-1}^2$, and $C=a_{m-1}a_{m-2}^2$. If $a_{m-1}=0$, then $C=0$ and $x=0$ is a critical point. If $a_{m-1}\ne0$ but $a_0=0$, alternative~\textup{(ii)} gives one critical point. When both coefficients are nonzero, $\rho=\beta/\alpha$, so
\[
\frac{C}{\beta\rho}=\frac{\alpha C}{\beta^2}
=\frac{a_0a_{m-2}^{\,2}}{a_{m-1}^{\,3}}.
\]
Alternative~\textup{(iii)} proves the criterion and the assertion about two zeros. The remaining conclusions follow from the last part of that theorem.
\end{proof}

\begin{rem}
\label{rem:simplest-cube}
Corollary~\textup{\ref{thm:linear-cube-criterion}} already yields counterexamples with no trace hypothesis on a parameter. Let $m\ge5$ be odd with $3\nmid m$ and put $A(x)=x+x^{q/4}+x^{q/2}$, so that $a_0=a_{m-2}=a_{m-1}=1$. On $K$ one has
\[
A(x)^4=x^4+x^{2^m}+x^{2^{m+1}}=x+x^2+x^4,
\]
and $a+a^2+a^4=0$ forces $a=0$ or $a^3+a+1=0$, since $X^3+X+1$ is irreducible over $\F_2$ and $3\nmid m$, the field $K$ contains no root of it, so $A$ has trivial kernel and is invertible. As $m$ is odd, $\Tr(a_0a_{m-2}^2/a_{m-1}^3)=\Tr(1)=1$, so $\mathcal R_q(A^3)$ is a reduced APN permutation of degree $3q/4$ all of whose coefficientwise lifts are permutations. For $m=5$ this is
\[
\mathcal R_{32}\bigl((x+x^8+x^{16})^3\bigr)=x+x^2+x^3+x^9+x^{10}+x^{18}+x^{24},
\]
a seven-term polynomial over $\F_2$ of degree $24$ whose derivative is $1+x^2+x^8$. Note that the same counterexample has circulated independently~\textup{\cite{StoneGH}}.
\end{rem}

\begin{proof}[Proof of Theorem~\textup{\ref{thm:binary-family}}]
We put $K=\F_q$ and let $H(x)=x+x^4+x^{16}$. This is an $\F_2$-linear map on $K$. If $H(a)=0$, then
\[
H(a)^4+H(a)=a^{64}+a=0.
\]
Thus $a$ is fixed by the sixth power of Frobenius. The fixed field in $K$ is $\F_{2^{\gcd(6,m)}}=\F_2$, since $m$ is odd and $3\nmid m$. This standard fixed-field description follows from the cyclic Galois group of a finite field~\cite{LidlNiederreiter97}. For $a\in\F_2$ we have $H(a)=a$, so $a=0$. Therefore $H$ has trivial kernel and is bijective.

On $K$, raising $H$ to the $q/8$-th power gives
\[
H(x)^{q/8}=x^{q/8}+x^{q/2}+x^{2q}
=x^{q/8}+x^{q/2}+x^2=A_m(x).
\]
The map $x\mapsto x^{q/8}$ is a Frobenius automorphism, so $A_m$ is also an invertible $\F_2$-linear map. Its fourth power has reduced representative
\[
B_m(x)=\mathcal{R}_q(A_m^4)(x)=x^8+x^{q/2}+x^2.
\]
Both $A_m$ and $B_m$ have degree $q/2$ with leading coefficient one. Their product therefore has degree $q$ with leading coefficient one, and
\begin{equation}\label{eq:binary-product}
\mathcal{R}_q(A_m^5)=A_m(x)B_m(x)+x^q+x.
\end{equation}
Expanding the product gives
\[
\begin{aligned}
A_mB_m={}&x^4+x^{10}+x^{q/2+2}
+x^{q/8+2}+x^{q/8+8}+x^{5q/8}\\
&+x^{q/2+2}+x^{q/2+8}+x^q.
\end{aligned}
\]
The two terms $x^{q/2+2}$ cancel. Substituting into \eqref{eq:binary-product} yields the polynomial $P_m$ in \eqref{eq:binary-family}. All its exponents are less than $q$, since $q\ge32$. It is consequently the reduced representative of $A_m^5$.

We apply Theorem~\ref{thm:gold-criterion} with $s=2$. Its coprimality hypothesis holds because $m$ is odd. Thus $P_m$ is an APN permutation. Every exponent in \eqref{eq:binary-family}, other than the initial exponent one, is even. Formal differentiation therefore gives the polynomial identity $P_m'=1$.  Here the identity $P_m'=1$ presupposes that no two of the seven  exponents coincide, which we now check. For $m=5$ the seven exponents are $1,4,10,6,12,20,24$, which are distinct. For $m\ge7$ we have $q/8\ge16>8$, whence
$$
1<4<10<q/8+2<q/8+8<q/2+8<5q/8 ,$$
the last inequality because $5q/8-(q/2+8)=q/8-8>0$, so the seven exponents are again distinct. We then apply Theorem~\ref{thm:gr-permutation-criterion} to obtain the lifting assertion.

For $m=5$, the seven exponents are $1,4,10,6,12,20,24$, all distinct, giving ordinary degree $24$. If $m\ge7$, then $q/8\ge16$. The largest exponent is $5q/8$, since $5q/8-(q/2+8)=q/8-8>0$, and it is also larger than each of the other exponents. Its coefficient is one, so the degree is $5q/8$.
\end{proof}

For example, at $m=5$ the theorem gives
$P_5(x)=x+x^4+x^6+x^{10}+x^{12}+x^{20}+x^{24}$.
This example has coefficients in $\F_2$ and constant formal derivative one. At $m=7$, the same construction gives $P_7(x)=x+x^4+x^{10}+x^{18}+x^{24}+x^{72}+x^{80}$. These are reduced representatives in the linear equivalence class of $x^5$, with ordinary degrees $24$ and $80$, respectively. Their lifted permutations are not asserted to be APN over the Galois rings. The APN statements concern their residue-field functions.

The cube construction gives a complementary family with a nonconstant derivative and an explicit trace certificate. It also produces a degree-$18$ example in dimension five. We retain the notation $F_w$ of Theorem~\ref{thm:trace-family} for this family, to distinguish it from $P_m$.

\begin{proof}[Proof of Theorem~\textup{\ref{thm:trace-family}}]
Let $A_w(x)=x+wT_m(x)$, with $T_m$ as in the statement. For $x\in K$, the value $T_m(x)$ lies in $\F_2$, and therefore
\[
T_m(A_w(x))=T_m(x)+T_m(w)T_m(x)=T_m(x).
\]
Applying $A_w$ twice gives $A_w(A_w(x))=x$, so $A_w$ is an invertible $\F_2$-linear map. Its linearized coefficients are $a_0=1+w$ and $a_i=w$ for $1\le i\le m-1$. Since $m$ is odd, $T_m(1)=1$, and the hypothesis $T_m(w)=0$ implies $w\ne1$. Thus $a_0a_{m-1}\ne0$, and
\[
T_m\!\left(\frac{a_0a_{m-2}^{\,2}}{a_{m-1}^{\,3}}\right)
=T_m(1+w^{-1})=1.
\]
Corollary~\ref{thm:linear-cube-criterion} shows that the reduced representative of $A_w^3$ is an APN permutation with nowhere-zero derivative.

It remains to identify that representative and the stated certificate. On $K$, the identities $T_m(x)^2=T_m(x)$ and $T_m(x)^3=T_m(x)$ give
\[
A_w(x)^3=x^3+w x^2T_m(x)+w^2xT_m(x)+w^3T_m(x).
\]
The right-hand side is the polynomial $F_w$ defined in the statement. Its degree is $2^{m-1}+2<2^m$, because its unique term of that degree has coefficient $w\ne0$. It is consequently the reduced representative. Every exponent has binary weight at most two, so the Boolean algebraic degree is at most two. It is exactly two because an affine function on a field with more than two elements has differential uniformity $q$, whereas $F_w$ is APN.

Since $T_m'(x)=1$, formal differentiation gives
$F_w'(x)=(1+w)x^2+w^2x+w^2T_m(x)+w^3$.
For $y=A_w(x)$, we have $y^2=x^2+w^2T_m(x)$, and hence
$F_w'(x)=(1+w)y^2+w^2y+w^3$.
We put $u_w=(1+w)/w^4$ and $z=(1+w)y/w^2$. Multiplying the last identity by $u_w$ gives $u_wF_w'(x)=z^2+z+1+w^{-1}$.
We take traces to obtain $T_m(u_wF_w'(x))=1$ for every $x\in K$, as asserted. In particular, the derivative never vanishes. The coefficientwise lifting conclusion follows from Theorem~\ref{thm:gr-permutation-criterion}.

We also obtain the cardinalities of the derivative fibers from this identity. Both $x\mapsto y$ and $y\mapsto z$ are bijections. The image of $z\mapsto z^2+z$ is the trace-zero subspace $\{c\in K:T_m(c)=0\}$, which has dimension $m-1$ over $\F_2$, and each of its elements has exactly two preimages. Multiplication by $u_w\ne0$ and translation by $1+w^{-1}$ preserve fiber cardinalities. Therefore the image of $F_w'$ has cardinality $q/2$, and each nonempty fiber contains two elements. The trace identity proves that zero does not belong to this image.
\end{proof}

The trace hypotheses in Theorem~\ref{thm:trace-family} hold in infinitely many dimensions. In the next corollary, we specify an algebraic parameter satisfying both trace conditions and prove that it remains admissible in infinitely many dimensions.

\begin{cor}\label{cor:infinite-dimensions}
Let $r\ge1$ be odd, put $m=5r$, and let $K=\F_{2^m}$. Choose $w$ in the subfield $\F_{32}\subseteq K$ satisfying $w^5+w^2+1=0$. Then $T_m(w)=T_m(w^{-1})=0$, and the polynomial $F_w$ of Theorem~\textup{\ref{thm:trace-family}} is a reduced APN permutation of degree $2^{m-1}+2$. Every coefficientwise lift permutes $\GR(2^k,m)$ for every $k>1$.
\end{cor}

\begin{proof}
The polynomial $p(X)=X^5+X^2+1$ has no root in $\F_2$, and its remainder modulo $X^2+X+1$, the only monic irreducible quadratic over $\F_2$, is $1$. A reducible polynomial of degree five has a factor of degree one or two, so $p$ is irreducible. Thus a root $w$ generates $\F_{32}$. The coefficient of $X^4$ in $p$ is zero, so $T_5(w)=0$. The reciprocal polynomial is $X^5p(X^{-1})=X^5+X^3+1$. It is the minimal polynomial of $w^{-1}$ and also has zero $X^4$-coefficient. Hence $T_5(w^{-1})=0$. We use here the standard expression of the field trace as the sum of the conjugates~\cite{LidlNiederreiter97}.

Since $5\mid m$, the field $K$ contains a subfield of order $32$. For $a$ in that subfield, $a^{2^5}=a$, and grouping the $5r$ terms in the absolute trace into $r$ blocks gives
\[
T_m(a)=\sum_{j=0}^{r-1}\sum_{i=0}^4a^{2^{5j+i}}=rT_5(a)=T_5(a),
\]
because $r$ is odd. Applying this to $w$ and $w^{-1}$ verifies the two hypotheses of Theorem~\ref{thm:trace-family}. That theorem gives the remaining conclusions.
\end{proof}

\begin{rem}\label{rem:scope}
The restriction $m=5r$ in Corollary~\textup{\ref{cor:infinite-dimensions}} comes from the choice of parameter, not from the method. What Corollary~\textup{\ref{thm:linear-cube-criterion}} requires in dimension $m$ is one invertible linearized $A$ with $a_0a_{m-1}\ne0$ and $\Tr(a_0a_{m-2}^2/a_{m-1}^3)=1$. Remark~\textup{\ref{rem:simplest-cube}} exhibits such an $A$ for every odd $m$ with $3\nmid m$, and Corollary~\textup{\ref{cor:infinite-dimensions}} covers $5\mid m$ and exhaustive search produces one in every odd dimension $m\le 21$, including $m=9$ and $m=21$. A uniform construction covering the remaining dimensions, namely the odd $m$ with $3\mid m$ and $5\nmid m$, is left open.
\end{rem}

For $m=5$, this construction gives a degree-$18$ reduced APN permutation over the same field as the degree-$24$ example in Theorem~\ref{thm:counterexample}. Explicitly, take $w=\theta$ with $\theta^5+\theta^2+1=0$ and $T_5(x)=x+x^2+x^4+x^8+x^{16}$ in the formula for $F_w$. The preceding algebraic arguments establish its APN property, its permutation property, and the absence of rational critical points. The program \texttt{verify\_gold\_families.py}, reproduced in Appendix~\ref{app:verification-code}, checks the finite examples and the general Gold-power criterion using binary polynomial arithmetic. These computations supplement the algebraic proofs and are not used to establish either infinite family.

\begin{cor}\label{cor:linear-noninvariance}
Let $m\ge5$ be odd and suppose either that $3\nmid m$ or that $5\mid m$. Then there are linearly equivalent APN permutations of $\F_{2^m}$ whose reduced representatives have opposite coefficientwise lifting behavior. In particular, existence of permutation lifts for reduced representatives is not invariant under invertible $\F_2$-linear input changes.
\end{cor}

\begin{proof}
If $3\nmid m$, Theorem~\ref{thm:binary-family} and its proof give linearly equivalent APN permutations represented by the reduced polynomials $x^5$ and $P_m$. Their formal derivatives are $x^4$ and $1$, respectively. The former vanishes at zero and the latter never vanishes, so Theorem~\ref{thm:gr-permutation-criterion} shows that the first polynomial admits no coefficientwise permutation lift, whereas every coefficientwise lift of the second is a permutation.

If $5\mid m$, write $m=5r$ with $r$ odd and choose $w$ as in Corollary~\ref{cor:infinite-dimensions}. The reduced polynomials $x^3$ and $F_w$ induce linearly equivalent APN permutations, since $F_w=A_w^3$ as functions and $A_w$ is an invertible $\F_2$-linear map. The derivative of $x^3$ vanishes at zero, while that of $F_w$ is nowhere zero. The same lifting criterion completes the proof.
\end{proof}

Scalar affine changes of coordinates preserve the derivative condition. For $h(x)=\alpha f(\beta x+\gamma)+\eta$ with $\alpha\beta\ne0$, no reduction is needed when $f$ is reduced, and $h'(x)=\alpha\beta f'(\beta x+\gamma)$. Such changes preserve the existence of rational critical points. The maps $A_m$ and $A_w$ above establish that the larger class of invertible $\F_2$-linear changes need not preserve it.

\section{Non-lifting families and degree restrictions}
\label{sec:examples}

The elementary monomial cases have rational critical points. If $f(x)=ax^e$ with $e>1$, then $f'(x)=ae x^{e-1}$. When $e$ is even the derivative is identically zero, while for odd $e$ it vanishes at~$0$. This covers all APN power permutations, including the Gold, Kasami, Welch, Niho, inverse, and Dobbertin families whenever their defining power map is a permutation. See~\cite{budaghyan,carlet} for these families. In particular, the reduced representative $x^{q-2}$ of the inverse function has identically zero derivative in characteristic two.

We obtain the cubic case directly from the bijectivity of Frobenius on a finite field~\cite{LidlNiederreiter97}.

\begin{lemma}
\label{lem:degree-3}
 Let $q=2^m$ and let $f\in\F_q[x]$ have degree $3$. Then $f$ has exactly one critical point in~$\F_q$.

\end{lemma}

\begin{proof}
We write $f(x)=ax^3+bx^2+cx+d$ with $a\neq 0$. In characteristic $2$, the formal derivative is $f'(x)=ax^2+c$. If $c=0$ then $f'(0)=0$. If $c\neq 0$, the equation $f'(x)=0$ reduces to $x^2=c/a$, which has a solution in $\F_q$ because the Frobenius $x\mapsto x^2$ is a bijection of $\F_q$. 
In both cases the critical point is unique.
\end{proof}

By Theorem~\ref{thm:gr-permutation-criterion}, no cubic permutation polynomial has a coefficientwise permutation lift to $\GR(2^k,m)$ for any $k>1$. The counterexample in Section~\ref{sec:counterexample} shows that the critical-point conclusion does not extend to all reduced APN permutation polynomials.

\begin{rem}\label{rem:known-degree-bounds}
We distinguish the preceding lifting obstructions from bounds that exclude APN behavior itself. We apply Bartoli--Fatabbi--Ghiandoni~\textup{\cite[Theorem~3.2]{BFG23}} in the polynomial case as follows. For $f\in\F_q[x]$ of degree $d\ge9$, put
\[
\Phi_f(X,Y,Z)=\frac{f(X)+f(Y)+f(Z)+f(X+Y+Z)}{(X+Y)(X+Z)(Y+Z)}.
\]
The quotient is a polynomial, since the numerator vanishes when any two variables coincide and the three distinct linear factors in the denominator are relatively prime. Suppose $\Phi_f$ is nonconstant and let $S_f\subseteq\A^3$ be its zero set. An irreducible component is a maximal nonempty irreducible closed subset. A closed subset is irreducible if it is not the union of two proper closed subsets. It is \emph{absolutely irreducible} if it remains irreducible over an algebraic closure. If $S_f$ has such a component defined over $\F_q$ and not contained in the union of the planes $X=Y$, $X=Z$, and $Y=Z$, then
\begin{equation}\label{eq:BFG-bound}
d<0.45q^{1/4}+0.5
\end{equation}
implies that $f$ is not APN. The cited theorem requires only a component with these properties, rather than absolute irreducibility of the whole surface. For polynomial functions it is obtained by taking the denominator of the rational function in~\textup{\cite[Theorem~3.2]{BFG23}} to be one.

Under the component hypothesis, an APN polynomial of degree $d\ge9$ must therefore satisfy $d\ge0.45q^{1/4}+0.5$. The exclusion in \eqref{eq:BFG-bound} rules out APN behavior itself, hence also APN permutations. It does not assert that an existing APN permutation has a rational critical point or fails to lift. For a fixed polynomial of odd degree outside the Gold numbers $2^s+1$ and the Kasami--Welch numbers $2^{2s}-2^s+1$, with $s\ge1$, Aubry--McGuire--Rodier~\textup{\cite[Theorem~2.3]{AMR10}} similarly exclude APN behavior over all sufficiently large finite extensions of its coefficient field. This concerns the same fixed polynomial on those extensions.

The polynomials in Theorems~\textup{\ref{thm:binary-family}} and~\textup{\ref{thm:trace-family}} have degrees that grow linearly with the field size, so they do not give an infinite family of fixed degree. In particular, bounds excluding APN behavior in a specified degree range do not provide a universal field-size threshold beyond which APN permutation polynomials fail to lift. We note that Corollary~\textup{\ref{cor:RS-refutation}} supplies liftable reduced APN permutations in arbitrarily large dimensions.
\end{rem}

\Needspace{9\baselineskip}

\section{Conclusion}
\label{sec:conclusion}

We disprove R\o njom--Sandrib's Conjecture~2 even when the residue polynomial is required to be reduced. We independently verify the degree-$24$ AI-generated example supplied by Zhang and Haobo~\cite{ZhangHaobo26}. We then construct two infinite families with algebraic proofs and obtain their coefficientwise permutation lifts by applying the standard Galois-ring criterion, which requires the formal derivative to have no zero in the residue field.

Our general Gold-power criterion determines exactly when the reduced representative of $A^{2^s+1}$ satisfies this derivative condition, where $A$ is an invertible linearized polynomial on $\F_{2^m}$ and $\gcd(s,m)=1$. The proof identifies the constant term contributed by reduction modulo $x^{2^m}-x$, which must be included before differentiating. We apply the criterion to obtain the seven-term family $P_m$ over $\F_2$, with $P_m'=1$, for every odd $m\ge5$ with $3\nmid m$. In the cube case, we obtain the trace family $F_w$ of degree $2^{m-1}+2$. Its derivative is nonconstant, and an explicit trace identity excludes zero from its image. A fixed parameter in $\F_{32}$ gives this family for every $m=5r$ with $r$ odd, including a degree-$18$ example in dimension five. Every coefficientwise lift of either family is a permutation at every ring level $k>1$.

By comparing these families with $x^5$ and $x^3$, we also show that coefficientwise liftability of reduced representatives is not invariant under invertible $\F_2$-linear input changes. The APN property is preserved by these changes, whereas the derivative condition can change after polynomial reduction. We complement this analysis with an exact trace-adjoint criterion for quadratic functions and a geometric interpretation through ramification at finite rational points. Known degree bounds for APN existence, including Bartoli--Fatabbi--Ghiandoni~\cite[Theorem~3.2]{BFG23}, remain compatible with our families. As explained in Remark~\ref{rem:known-degree-bounds}, those bounds do not obstruct lifting of an existing APN permutation.

The results leave open the least possible ordinary degree of a liftable reduced APN permutation and the distribution of liftable representatives within a Gold equivalence class. The coefficient conditions in Theorem~\ref{thm:gold-criterion} provide a starting point for counting such representatives and for seeking analogous constructions from other APN power families. For the trace family, one can also ask how many parameters satisfy both trace conditions and how field automorphisms act on them. These questions concern the reduced residue polynomial. Once its derivative has no rational zero, neither the ring level nor the higher coefficient digits impose a further permutation obstruction.

\vskip.3cm
\paragraph{Acknowledgments}
We thank Wenlin Zhang and Haobo for sharing the counterexample and its Lean~4 formalization through the TrureTuring project.

\bibliographystyle{plain}

\appendix
\section{Python verification code}
\label{app:verification-code}

We reproduce the complete program \texttt{\small verify\_gold\_families.py} below. It uses only the Python standard library and requires Python~3.8 or later and can be run by
\texttt{\small python3 verify\_gold\_families.py.}
Run Python without the \texttt{-O} option, since the verification conditions are implemented as assertions.

The program represents field elements by integers whose binary digits are the coefficients in a polynomial basis. Addition is bitwise exclusive-or, and multiplication is polynomial multiplication modulo the specified irreducible polynomial. A dictionary maps each exponent to its coefficient. Positive exponents are reduced modulo $2^m-1$ to the range $1,\ldots,2^m-1$ before formal differentiation.

The checks cover the supplied degree-$24$ example, all ten admissible trace parameters over $\F_{32}$, and the seven-term family in dimensions $5$, $7$, and $11$. The APN property is checked exhaustively for the supplied example, the ten trace examples, and the seven-term examples in dimensions $5$ and $7$. In dimension $11$, the program checks the reduced polynomial, the linear permutation, the Gold-power identity, and the derivative identity. It also tests the cube and general Gold-power criteria on all $168$ invertible linear maps over $\F_8$ and on deterministic samples over $\F_{32}$, using the fixed seed \texttt{20261015}. These finite checks supplement the algebraic proofs in Section~\ref{sec:linear-family}. The permutation-lifting conclusions follow from Theorem~\ref{thm:gr-permutation-criterion}.

\begin{lstlisting}[style=verificationpython]
#!/usr/bin/env python3
"""Finite checks accompanying the liftable reduced APN constructions.

Run with Python 3. No third-party packages are required.
The infinite-family results have algebraic proofs in the manuscript.
"""
from collections import Counter
from itertools import product
import random


class BinaryField:
    def __init__(self, m, modulus):
        self.m, self.modulus, self.q = m, modulus, 1 << m

    def mul(self, a, b):
        result = 0
        while b:
            if b & 1:
                result ^= a
            b >>= 1
            a <<= 1
            if a & self.q:
                a ^= self.modulus
        return result

    def power(self, a, exponent):
        result = 1
        while exponent:
            if exponent & 1:
                result = self.mul(result, a)
            a = self.mul(a, a)
            exponent >>= 1
        return result

    def inv(self, a):
        assert a
        return self.power(a, self.q - 2)

    def trace(self, a):
        result = 0
        for _ in range(self.m):
            result ^= a
            a = self.mul(a, a)
        assert result in (0, 1)
        return result

    def evaluate(self, coefficients, x):
        result = 0
        for exponent in range(max(coefficients, default=0), -1, -1):
            result = self.mul(result, x) ^ coefficients.get(exponent, 0)
        return result


def derivative(coefficients):
    return {e - 1: a for e, a in coefficients.items() if e & 1 and a}


def add_term(coefficients, exponent, coefficient):
    coefficients[exponent] = coefficients.get(exponent, 0) ^ coefficient


def reduced_cube(field, coefficients):
    # Expand A(x)A(x)^2 directly, then reduce every monomial modulo x^q-x.
    result = {}
    for i, a in enumerate(coefficients):
        for j, b in enumerate(coefficients):
            exponent = (1 << i) + (1 << (j + 1))
            exponent = 1 + (exponent - 1) % (field.q - 1)
            add_term(result, exponent, field.mul(a, field.mul(b, b)))
    return {e: a for e, a in result.items() if a}


def check_apn_permutation(field, coefficients):
    values = [field.evaluate(coefficients, x) for x in range(field.q)]
    assert len(set(values)) == field.q
    distribution = Counter()
    for t in range(1, field.q):
        row = Counter(values[x ^ t] ^ values[x] for x in range(field.q))
        assert max(row.values()) == 2
        distribution.update(row.get(b, 0) for b in range(field.q))
    return distribution


def check_general_criterion(field, vectors):
    tested = liftable = 0
    for a in vectors:
        linear = {1 << i: coefficient for i, coefficient in enumerate(a)}
        images = [field.evaluate(linear, x) for x in range(field.q)]
        if len(set(images)) != field.q:
            continue
        tested += 1
        f = reduced_cube(field, a)
        df = derivative(f)
        actual = [field.evaluate(df, x) for x in range(field.q)]
        constant = field.mul(a[-1], field.mul(a[-2], a[-2]))
        for x, y in enumerate(images):
            expected = (field.mul(a[0], field.mul(y, y))
                        ^ field.mul(field.mul(a[-1], a[-1]), y) ^ constant)
            assert actual[x] == expected
        predicted = False
        if a[0] and a[-1]:
            numerator = field.mul(a[0], field.mul(a[-2], a[-2]))
            argument = field.mul(numerator, field.inv(field.power(a[-1], 3)))
            predicted = field.trace(argument) == 1
            assert actual.count(0) == (0 if predicted else 2)
        assert (0 not in actual) == predicted
        liftable += predicted
    return tested, liftable


def verify_previous_constructions():
    field = BinaryField(5, 0b100101)  # X^5+X^2+1
    supplied = dict(zip(
        [0, 1, 2, 3, 4, 5, 6, 8, 9, 10, 12, 16, 17, 18, 20, 24],
        [30, 30, 31, 12, 3, 30, 29, 9, 27, 7, 5, 9, 26, 11, 28, 6]))
    expected_values = [30, 1, 14, 20, 3, 29, 11, 16, 28, 10, 27, 8,
                       26, 13, 5, 23, 22, 15, 24, 4, 7, 31, 17, 12,
                       25, 9, 0, 21, 19, 2, 18, 6]
    assert [field.evaluate(supplied, x) for x in range(32)] == expected_values
    assert check_apn_permutation(field, supplied) == Counter({0: 496, 2: 496})
    ds = derivative(supplied)
    assert ds == {0: 30, 2: 12, 4: 30, 8: 27, 16: 26}
    assert all(field.trace(field.mul(14, field.evaluate(ds, x))) == 1
               for x in range(32))
    assert len({field.evaluate(ds, x) for x in range(32)}) == 16
    print('Supplied degree-24 example: permutation, APN, trace certificate verified.')

    parameters = [w for w in range(1, 32)
                  if field.trace(w) == field.trace(field.inv(w)) == 0]
    assert parameters == [2, 4, 9, 11, 13, 15, 16, 18, 27, 31]
    for w in parameters:
        w2, w3 = field.power(w, 2), field.power(w, 3)
        f = {3: 1}
        for i in range(5):
            add_term(f, 2 + (1 << i), w)
            add_term(f, 1 + (1 << i), w2)
            add_term(f, 1 << i, w3)
        f = {e: a for e, a in f.items() if a}
        assert max(f) == 18
        assert f == reduced_cube(field, [1 ^ w, w, w, w, w])
        assert check_apn_permutation(field, f) == Counter({0: 496, 2: 496})
        df = derivative(f)
        certificate = field.mul(1 ^ w, field.inv(field.power(w, 4)))
        values = [field.evaluate(df, x) for x in range(32)]
        assert 0 not in values
        assert len(set(values)) == 16 and set(Counter(values).values()) == {2}
        for x, value in enumerate(values):
            assert field.trace(field.mul(certificate, value)) == 1
            y = x ^ field.mul(w, field.trace(x))
            assert field.evaluate(f, x) == field.power(y, 3)
    print('All 10 trace parameters over F_32: degree-18 construction verified.')

    small = BinaryField(3, 0b1011)  # X^3+X+1
    tested, liftable = check_general_criterion(small, product(range(8), repeat=3))
    assert tested == 168
    print(f'General criterion over F_8: all {tested} invertible maps checked; '
          f'{liftable} yield nowhere-zero derivatives.')
    rng = random.Random(20261015)
    vectors = [tuple(rng.randrange(32) for _ in range(5)) for _ in range(400)]
    tested, liftable = check_general_criterion(field, vectors)
    print(f'General criterion over F_32: {tested} invertible maps in deterministic '
          f'sample checked; {liftable} yield nowhere-zero derivatives.')
    print('All checks passed. These computations supplement the algebraic proofs.')





def reduced_gold_power(field, coefficients, s):
    # Expand A*A^(2^s) and reduce positive exponents modulo q-1.
    result = {}
    for i, a in enumerate(coefficients):
        for j, b in enumerate(coefficients):
            exponent = (1 << i) + (1 << (j + s))
            exponent = 1 + (exponent - 1) % (field.q - 1)
            add_term(result, exponent, field.mul(a, field.power(b, 1 << s)))
    return {e: a for e, a in result.items() if a}


def check_gold_criterion(field, vectors):
    from math import gcd
    maps = cases = constant_nonzero = 0
    for a in vectors:
        linear = {1 << i: coefficient for i, coefficient in enumerate(a)}
        images = [field.evaluate(linear, x) for x in range(field.q)]
        if len(set(images)) != field.q:
            continue
        maps += 1
        for s in range(1, field.m):
            if gcd(s, field.m) != 1:
                continue
            cases += 1
            f = reduced_gold_power(field, a, s)
            df = derivative(f)
            actual = [field.evaluate(df, x) for x in range(field.q)]
            alpha = a[0]
            beta = field.power(a[field.m - s], 1 << s)
            constant = field.mul(a[-1], field.power(a[field.m - 1 - s], 1 << s))
            for x, y in enumerate(images):
                expected = (field.mul(alpha, field.power(y, 1 << s))
                            ^ field.mul(beta, y) ^ constant)
                assert actual[x] == expected
            if alpha == beta == 0:
                assert df == ({0: constant} if constant else {})
                predicted = constant != 0
                constant_nonzero += predicted
                assert actual.count(0) == (0 if predicted else field.q)
            elif alpha == 0 or beta == 0:
                predicted = False
                assert actual.count(0) == 1
            else:
                ratio = field.mul(beta, field.inv(alpha))
                rho = field.power(ratio, pow((1 << s) - 1, -1, field.q - 1))
                argument = field.mul(constant, field.inv(field.mul(beta, rho)))
                predicted = field.trace(argument) == 1
                assert actual.count(0) == (0 if predicted else 2)
            assert (0 not in actual) == predicted
    return maps, cases, constant_nonzero


def sparse_value(field, coefficients, x):
    value = 0
    for e, a in coefficients.items():
        value ^= field.mul(a, field.power(x, e))
    return value


def check_binary_family(field, full_apn):
    m, q = field.m, field.q
    assert m >= 5 and m % 2 and m % 3
    f = {e: 1 for e in [1, 4, 10, q // 8 + 2, q // 8 + 8,
                        5 * q // 8, q // 2 + 8]}
    assert len(f) == 7
    assert max(f) == (24 if m == 5 else 5 * q // 8)
    assert derivative(f) == {0: 1}
    a = [0] * m
    for i in [1, m - 3, m - 1]:
        a[i] = 1
    assert f == reduced_gold_power(field, a, 2)
    images = []
    values = []
    for x in range(q):
        h = x ^ field.power(x, 4) ^ field.power(x, 16)
        y = field.power(x, 2) ^ field.power(x, q // 8) ^ field.power(x, q // 2)
        assert field.power(h, q // 8) == y
        images.append(y)
        value = sparse_value(field, f, x)
        assert value == field.power(y, 5)
        values.append(value)
    assert len(set(images)) == len(set(values)) == q
    if full_apn:
        for t in range(1, q):
            row = Counter(values[x ^ t] ^ values[x] for x in range(q))
            assert max(row.values()) == 2
    return max(f)


def main():
    verify_previous_constructions()
    for m, modulus in [(5, 0b100101), (7, 0b10000011), (11, (1 << 11) | 5)]:
        field = BinaryField(m, modulus)
        degree = check_binary_family(field, full_apn=(m <= 7))
        print(f'Seven-term family over F_2^{m}: degree {degree}, reduction, '
              'linear permutation, Gold identity and derivative 1 verified.'
              + (' Complete DDT checked.' if m <= 7 else ''))
    field = BinaryField(3, 0b1011)
    maps, cases, constants = check_gold_criterion(field, product(range(8), repeat=3))
    assert maps == 168 and cases == 336
    print(f'Gold criterion over F_8: all {maps} invertible maps, {cases} cases checked.')
    field = BinaryField(5, 0b100101)
    rng = random.Random(20261015)
    vectors = [tuple(rng.randrange(32) for _ in range(5)) for _ in range(400)]
    vectors.append((0, 1, 1, 0, 1))  # Constant-derivative construction at s=2.
    maps, cases, constants = check_gold_criterion(field, vectors)
    assert constants > 0
    print(f'Gold criterion over F_32: {maps} invertible maps, {cases} cases checked, '
          f'including {constants} nonzero constant-derivative cases.')
    print('All Gold-family checks passed. Infinite claims are proved in the manuscript.')


if __name__ == '__main__':
    main()
\end{lstlisting}

\end{document}